\documentclass[amssymb]{article}
\usepackage[english]{babel}
\usepackage{amsmath,amssymb,url}
\usepackage{amsfonts}
\usepackage{amsthm}
\usepackage[T2A]{fontenc}
\usepackage[all]{xy}
\usepackage{mathrsfs}
\usepackage{multicol}

\ifx\pdfoutput\undefined
\usepackage{graphicx}
\else
\usepackage[pdftex]{graphicx}
\fi \graphicspath{{noiseimages/}}
\usepackage{wrapfig}
\begin{document}
\newtheorem{thm}{Theorem}
\newtheorem{lm}{Lemma}
\newtheorem{cor}{Corollary}
\newtheorem{prop}{Proposition}
\newtheorem{claim}{Claim}
\newtheorem{note}{Remark}
\theoremstyle{definition}
\newtheorem{df}{Definition}
\let\emptyset\varnothing
\let\leq\leqslant
\let\geq\geqslant

\title{Functional Galois connections and a classification of symmetric conservative clones with a finite carrier}
\author{Nikolay~L.~Polyakov\\ \small Higher School of Economics, niknikols0@gmail.com}
\date{}
\maketitle

\begin{abstract}
We propose a classification of symmetric conservative clones with a
finite carrier. For the study, we use the functional Galois
connection $(\mathrm{Inv}_Q, \mathrm{Pol}_Q)$, which is a natural
modification of the connection $(\mathrm{Inv}, \mathrm{Pol})$ based
on the preservation relation between functions $f$ on a set $A$ (of
all finite arities) and sets of functions $h\in A^Q$ for an
arbitrary set $Q$.
\end{abstract}


\section{Introduction}\label{sec:intro}

Clones of conservative functions on an arbitrary set $A$ are a
naturally generalization of closed classes of Boolean functions
preserving $\mathbf{0}$ and $\mathbf{1}$. Some important information
about conservative clones can be found in the papers
\cite{Csakany86} and \cite{Jezek95}. In 2005, S. Shelah
\cite{Shelah2005} found an unexpected application of conservative
clones to Computational Social Choice. Using the methods of
\cite{Shelah2005}, a complete classification of symmetric classes of
selection functions with the Arrow property was obtained in
\cite{Polyakov2014}. A further development of this approach requires
an explicit classification of symmetric conservative clones with a
finite carrier, as well as a description of the corresponding
fragment of the functional Galois connection $(\mathrm{Inv}_Q,
\mathrm{Pol}_Q)$. Functional Galois connection $(\mathrm{Inv}_Q,
\mathrm{Pol}_Q)$ is a natural modification of the connection
$(\mathrm{Inv}, \mathrm{Pol})$ based on the preservation relation
between functions $f$ on a set $A$ (of all finite arities) and sets
of functions $h\in A^Q$ for an arbitrary set $Q$. The
$(\mathrm{Inv}_Q, \mathrm{Pol}_Q)$-connection is a convenient tool
for studying closed classes of discrete functions since, on the one
hand, it is easily transformed into a $(\mathrm{Inv},
\mathrm{Pol})$-connection and, on the other hand, it is closely
related to the functional closure.

\section{Notation and basic definitions}
\par Let $A$ be an arbitrary non-empty set. For
any set $Q$ the symbol $A^Q$ denotes the set of all function
$f\colon Q\to A$. Elements of the Cartesian power $A^{n}$,
\mbox{$n<\omega$}, are identified with sequences
$(a_0,a_1,\ldots,a_{n-1})$, $a_i\in A$, $i<n$, i.e. with functions
$\mathbf{a}\colon n=\{\,0,1,\ldots,n-1\,\}\to A$; therefore, for any
sequence \mbox{$\mathbf{a}\in A^{<\omega}$} the standard notations
$\mathrm{dom}\, \mathbf{a}$ and $\mathrm{ran}\, \mathbf{a}$ denote
the domain  and the range of $\mathbf{a}$, respectively. The set
$\{\,\mathbf{a}\in A^n\colon  |\mathrm{ran}\, \mathbf{a}|=k\,\}$ is
denoted by~$A^n_k$. In a natural sense, we use the notations
$A^{n}_{<m}= \bigcup_{k<m}A^{n}_{k}$, $A^{n}_{\leq m}=
\bigcup_{k\leq m}A^{n}_{k}$, etc. Usually, when writing sequences we
omit  special characters, i.e. instead of $(a_0,a_1,\ldots,a_{n-1})$
we write   $a_0a_1\ldots a_{n-1}$.

\par The symbol $\mathcal{O}(A)$
denotes the set of all functions over $A$ (of all arities), i.e.
$\mathcal{O}(A)=\bigcup_{n<\omega}A^{A^n}$. For any $n<\omega$ and
$\mathcal{F}\subseteq \mathcal{O}(A)$, the symbol
$\mathcal{F}_{[n]}$ denotes the set of all $n$-ary functions in
$\mathcal{F}$, i.e. $\mathcal{F}_{[n]}=\mathcal{F}\cap A^{A^n}$. For
every natural number  $n$, function $f\in \mathcal{O}(A)_{[n]}$ and
functions  $h_0,h_1, \ldots, h_{m-1}\in A^Q$ the symbol $f(h_0,h_1,
\ldots, h_{m-1})$ denotes the \emph{composition} of these functions
defined by
\[
f(h_0,h_1, \ldots, h_{m-1})(q)=f(h_0(q)h_1(q) \ldots h_{m-1}(q))
\]
for each $q\in Q$. We hope that  it is always clear from the context
what the expression $f(a_0a_1\ldots a_{n-1})$ denotes, the value of
the function on the sequence $a_0a_1\ldots a_{n-1}$ or the
composition $f$ and $a_0a_1\ldots a_{n-1}$. For example, for any
$f\colon A\to A$ and $\mathbf{a}=a_0a_1\ldots a_{n-1}\in A^n$, we
have $f(\mathbf{a})=f(a_0)f(a_1)\ldots f(a_{n-1})$ (if $n = 1$, we
identify an element $a_0\in A$ and the corresponding one-element
sequence).

\par The function $f\in\mathcal{O}(A)_{[n]}$  which
assigns to each sequence $\mathbf{a}\in A^{n}$,
$\mathbf{a}=a_0a_1\ldots a_{n-1}$, the element $a_i$ for some fixed
number  $i<n$, is called the $n$-ary $i$-th \emph{projection} and is
denoted by~$e^{n}_{i}$. The set of all projection $e\in
\mathcal{O}(A)$ is denoted by $\mathcal{E}(A)$. In accordance with
the standard notation, we sometimes write $x_i$ instead of~$e^n_i$.
For example, for any binary function $f$, we write $f(x_2, x_1)$
instead of $f(e^2_2, e^2_1)$.

\par The
restriction of a function $h\colon Q\to A$ to a set $P$ is denoted
by $h|_P$, i.e. $h|_P=h\cap (P\times A)$ (we do not assume that
$P\subseteq Q$). The symbol $h|^P$ denotes the set $\{\,f\in
A^{P\cup\mathrm{dom}\, f}\colon  f|_{\mathrm{dom}\, h}=h\,\}.$ For
any set $H$ of functions (perhaps with different domains), we denote
$H|_P=\{\,h|_P\colon h\in H\,\}$ and $H|^P=\bigcup_{h\in H}h|^P$.

\par For any $H\subseteq A^Q$, $\mathbf{h}=h_0h_1\ldots h_{n-1}\in H^{n}$ and $q\in Q$ we denote $H(q)=\{\,h(q)\colon h\in H\,\}$
and $\mathbf{h}(q)=h_0(q)h_1(q)\ldots h_{n-1}(q)$.

\par The set of
all subsets of $A$ is denoted by~$\mathscr P(A)$. The set of all
$k$-element subsets of $A$ is denoted by~$[A]^k$, i.e.
$[A]^k=\{\,B\subseteq A\colon  |B|=k\,\}$. The symbol~$S_A$ denotes
the set of all permutations of $A$.

\begin{df}
A set $\mathcal{F}\subseteq \mathcal{O}(A)$ is called a \emph{clone}
(with the carrier $A$) if it is closed with respect to composition
and contains all projections.
\end{df}

\begin{df} A \emph{natural isomorphism} from clone $\mathcal{F}$
to clone $\mathcal{G}$ with carriers $A$ and $B$, respectively, is a
pair of one-to-one functions $\sigma\colon  A\to B$ and $\tau\colon
\mathcal{F}\to \mathcal{G}$ for which
$$
f\in \mathcal{F}_{[n]} \Rightarrow \tau(f)\in \mathcal{G}_{[n]}
$$
for all $f\in \mathcal{F}$ and natural number $n$, and
$$
\sigma(f(\mathbf{a}))=\tau(f)(\sigma(\mathbf{a}))
$$
for all $f\in \mathcal{F}$ and $\mathbf{a}\in \mathrm{dom}\, f$.
\par Clones $\mathcal{F}$ and $\mathcal{G}$ are \emph{naturally isomorphic} if
there is a natural isomorphism from $\mathcal{F}$ to $\mathcal{G}$.
\end{df}

\begin{df}
A clone $\mathcal{F}\subseteq \mathcal{O}(A)$ is called
\emph{symmetric} if for any function $f\in \mathcal{O}(A)$ and
permutation $\sigma$ of $A$
$$
f\in \mathcal{F}\Rightarrow f_{\sigma}\in \mathcal{F},
$$
where for any $f\in \mathcal{O}(A)$ the function $f_{\sigma}$ is
defined by
$$
f_{\sigma}(\mathbf{a})=\sigma^{-1}(f(\sigma(\mathbf{a})))
$$
for all $\mathbf{a} \in \mathrm{dom}\, f$.
\end{df}


\section{Galois connections $(\mathrm{Inv}_Q, \mathrm{Pol}_Q)$}
\par The Galois connection  $(\mathrm{Inv}, \mathrm{Pol})$ is one of the basic concepts in the theory of discrete
functions, see \cite{Poshel}, \cite{Lau}. Galois connection
$(\mathrm{Inv}, \mathrm{Pol})$  allows one to characterize clones
using their invariant sets. Recall that a function \mbox{$f\in
\mathcal{O}(A)_{[n]}$} \emph{preserves} a predicate $P\subseteq A^m$
if for all $a^0_0a^0_1\ldots a^0_{m-1}$, $a^1_0a^1_1\ldots
a^1_{m-1}$, $\ldots$, $a^{n-1}_0a^{n-1}_1\ldots a^{n-1}_{m-1}$
from~$P$ we have

\[
f(a^0_0a^1_0\ldots a^{n-1}_0)f(a^0_1a^1_1\ldots a^{n-1}_1)\ldots
f(a^0_{m-1}a^1_{m-1}\ldots a^{n-1}_{m-1})\in P.
\]

For any set $\mathcal{F}\subseteq \mathcal{O}(A)$ the set of all
predicates $P$ such that any function $f\in \mathcal{F}$ preserves
$P$ is denoted by $\mathrm{Inv}\, \mathcal{F}$. In the opposite
direction, for any set $\mathbb{P}$ of predicates the set of all
functions $f\in  \mathcal{O}(A)$ such that $f$ preserves any
predicate $P\in \mathbb{P}$ is denoted by $\mathrm{Pol}\,
\mathbb{P}$. The pair $(\mathrm{Inv}, \mathrm{Pol})$ is a Galois
connection between boolean lattices $\mathscr
P\left(\bigcup_{m<\omega}\mathscr P(A^m)\right)$ and $\mathscr
P(\mathcal{O}(A))$. Any Galois closed set $\mathcal{F}\in
\mathcal{O}(A)$ is a clone. If $A$ is a finite set, the class of
Galois closed sets $\mathcal{F}\in \mathscr P(\mathcal{O}(A))$
coincides with the class of all clones with the carrier~$A$.

\par However, it will be more convenient for us to use the concept of  functional Galois
connections.

\begin{df}
Let $A$ and $Q$ be non-empty sets. A function  $f\in
\mathcal{O}(A)_{[n]}$ \emph{preserves} a set $H\subseteq A^Q$  if
for all $h_{0},h_{1}\ldots,h_{n-1}\in H$ the set~$H$ contains the
function $f(h_0,h_{1}\ldots, h_{n-1})$.
\par For any set $\mathcal{F}\subseteq \mathcal{O}(A)$ the set of all
sets $H\subseteq A^Q$ such that any function $f\in \mathcal{F}$
preserves $H$ is denoted by $\mathrm{Inv}_Q\, \mathcal{F}$. Any set
$H\in \mathrm{Inv}_Q\, \mathcal{F}$ is called a
\mbox{\emph{$Q$-invariant set}} of $\mathcal{F}$.
\par For any set
$\mathbb{H}\subseteq \mathscr P(A^Q)$ the set of all functions $f\in
\mathcal{O}(A)$ such that $f$ preserves any set $H\in \mathbb{H}$ is
denoted by $\mathrm{Pol}_Q\, \mathbb{D}$.
\end{df}

\begin{prop}\label{InvPolprop} For all sets $A\neq\emptyset$, $\mathcal{F}\subseteq \mathcal{O}(A)$,
$Q$, $H, H'\subseteq A^Q$, $P$ and function $f\colon P\to Q$
\begin{enumerate}
\item $H\in \mathrm{Inv}_{Q} \mathcal{F}\Rightarrow \{\,h(f)\colon h\in H\,\}\in \mathrm{Inv}_{P}
\mathcal{F}$,
\item $H\in \mathrm{Inv}_{Q} \mathcal{F}\Rightarrow H|_{P}\in \mathrm{Inv}_{P\cap Q}
\mathcal{F}$,
\item $H\in \mathrm{Inv}_{Q} \mathcal{F}\Rightarrow H|^{P}\in \mathrm{Inv}_{P\cup Q}
\mathcal{F}$,
\item $H,H'\in \mathrm{Inv}_{Q} \mathcal{F}\Rightarrow H\cap H'\in \mathrm{Inv}_{Q}
\mathcal{F}$,
\end{enumerate}
\end{prop}
\begin{proof} By a direct verification.
\end{proof}

\par There is a simple relationship between these two preservation
relations. Consider an $m$-ary predicate $P$ over $A$ as a set of
functions $\mathbf{a}\colon m\to A$. It is easy to verify that for
any function \mbox{$f\in \mathcal{O}(A)$}, $P\in
\mathrm{Inv}\,\{\,f\,\}$ if and only if $P\in
\mathrm{Inv}_m\,\{\,f\,\}$. On the other hand, suppose that $Q$ is a
finite set of cardinality $m$, and let some numbering $Q=\{\,q_0,
q_1, \ldots, q_{m-1}\,\}$ be fixed. It easy to check that a function
\mbox{$f\in \mathcal{O}(A)$} preserves a set $H\subseteq A^Q$ if and
only if $f$ preserves the predicate $P=\{\,h(q_0)h(q_1)\ldots
h(q_{m-1})\colon  h\in H\,\}$. These arguments allow us immediately
to obtain some statements for $(\mathrm{Inv}_Q, \mathrm{Pol}_Q)$
connection. In particular, the following proposition holds.

\begin{prop}
For any non-empty sets $A$ and $Q$ the pair $(\mathrm{Inv}_Q,
\mathrm{Pol}_Q)$ is a Galois connection between the Boolean latices
$\mathscr P(\mathscr P(A^Q))$ and $\mathscr P (\mathcal{O}(A))$. Any
Galois closed set $\mathcal{F}\in \mathcal{O}(A)$ is a clone.
\end{prop}

\par Also note that for any clone $\mathcal{F}\subseteq
\mathcal{O}(A)$, $\mathcal{F}_{[n]}\in \mathrm{Inv}_{[A^n]}\,
\mathcal{F}$. It immediately follows that each clone $\mathcal{F}$
with a finite carrier is uniquely characterized by the set
$\mathrm{Inv}\, \mathcal{F}$. Thus, a family of Galois connections
$(\mathrm{Inv}_Q, \mathrm{Pol}_Q)$ unites the concepts of invariant
sets and functional closure. In some papers, other Galois
connections $(\mathrm{Inv}_Q, \mathrm{Pol}_Q)$ are considered. E.g.,
the case $Q=[A]^r$ is studied in \cite{Shelah2005},
\cite{Polyakov2014}.


\section{Decomposition theorems}
\par Now we will show that under certain conditions, the set $\mathrm{Inv}_Q \mathcal{F}$
is arranged quite simply. Let $A$, $Q$ be arbitrary sets, and let
$H\subseteq A^Q$ and $\mathscr R\subseteq \mathscr P(Q)$. The set
$$
H_{(\mathscr{R})}=\{\,(H|_R)|^Q\colon  R\in \mathscr R\,\}
$$
is called a \emph{decomposition} of $H$ over $\mathscr R$. It is
easy to verify that the following proposition is true.

\begin{prop}\label{decomp}
For all $H\subseteq A^Q$ and $\mathscr R\subseteq \mathscr P(Q)$
\begin{enumerate}
\item $H\subseteq \bigcap H_{(\mathscr {R})}$,
\item $\left(\bigcap  H_{(\mathscr R)}\right)|_R=H|_R$ for any set $R\in \mathscr
R$.
\end{enumerate}
\end{prop}

\begin{df}  A set $H\subseteq A^Q$ is \emph{decomposable} over $\mathscr
R$ if $H=\bigcap H_{(\mathscr R)}$.
\end{df}

We show that clones satisfying the conditions $\Delta^\partial$,
$\Delta^e_r$, and $\Delta^2$ defined below have $Q$-invariant sets
that are decomposable over non-trivial sets~$\mathscr R$.

\begin{df}
Let $\mathcal{F}$ be a clone with a carrier $A$ and $n$ a natural
number, $n\geq 2$. The clone $\mathcal{F}$ satisfies the condition
\begin{itemize}
\item $\Delta^{s}_{n}$ if there is a natural number $i<n$ such that
for all $\mathbf{a}\in A^{n}_{n}$ and $a\in \mathrm{ran}\,
\mathbf{a}$ there is a function $s\in \mathcal{F}_{[n]}$ for which
\[
s(\mathbf{a})=a\,\,\text{and}\,\,s(\mathbf{x})=x_i\,\,\text{for all
$\mathbf{x}=x_0x_1\ldots x_{n-1}\in A^{n}_{<n}$};
\]
\item $\Delta^{\partial}$ if for all $\mathbf{a}\in A^{3}_{3}$ and $a\in \mathrm{ran}\,
\mathbf{a}$ there is a function $\partial\in \mathcal{F}_{[3]}$ for
which
\[
\partial(\mathbf{a})=a\,\,\text{and}\,\, \partial(xxy)=\partial(xyx)=\partial(yxx)=x\,\,\text{for all $x,y\in
A$};
\]
\item $\Delta^{2}$ if for all $\mathbf{a}, \mathbf{b} \in A^{2}_{2}$ such that $\mathrm{ran}\, \mathbf{a}\neq \mathrm{ran}\, \mathbf{b}$, and for all $a\in \mathrm{ran}\,
\mathbf{a}$, and $b\in \mathrm{ran}\, \mathbf{b}$ there is a
function $w\in \mathcal{F}_{[2]}$ for which
\[
w(\mathbf{a})=a,\,\,w(\mathbf{b})=b\,\,
\text{and}\,\,w(xx)=x\,\,\text{for all $x\in A$}.
\]
\end{itemize}
\end{df}

\begin{df} Let $A$, $B\subseteq A$, $Q$, $H\subseteq A^Q$ be non-empty
sets and $n$ a natural number. We denote by
\begin{enumerate}
\item[(1)] $[Q]^{2,0}_{H}$ the set of all sets $P=\{\,p,q\,\}\in [Q]^2$
such that there is a permutation $\sigma\in S_A$ for which
$h(q)=\sigma (h(p))$ for all $h\in H$,
\item[(2)] $[Q]^{2,id}_{H}$ the set of all sets $P=\{\,p,q\,\}\in [Q]^2$
such that $h(q)=h(p)$ for all $h\in H$,
\item[(3)] $[Q]^{2,1}_{H}$ the set of all sets $P=\{\,p,q\,\}\in [Q]^2$
such that there are $a,b\in A$ for which $h(p)=a\vee h(q)=b$ for all
$h\in H$,
\item[(4)] $Q^{(n)}_{H}$ the set $\{\,q\in Q\colon  |H(q)|<n\,\}$,
\item[(5)] $Q^{[B]}_{H}$ the set $\{\,q\in Q\colon  |H(q)|\subseteq B\,\}$.
\end{enumerate}
\end{df}

\par We will use the following technical definition.

\begin{df}
Let $H\subseteq A^Q$, $p,q\in Q$ and $a\in H(p)$. We say that $H$
\emph{weakly separates} $p$ \emph{from} $q$ \emph{at the point} $a$
if $H$ contains functions $h_1$ and $h_2$ such that
\[
h_1(p)=h_2(p)=a\,\,\text{and}\,\,h_1(q)\neq h_2(q)
\]
If $H$ weakly separates $p$ \emph{from} $q$  or $q$ \emph{from} $p$
at least at one point, we will simply say that it \emph{weakly
separates} $p$ \emph{and} $q$.
\par We say that $H$ \emph{strongly
separates}  $p$ \emph{from} $q$ \emph{at the point} $a$ if for each
$b\in H(q)$, $H$ contains a function $h$ such that
\[
h(p)=a\,\,\text{and}\,\,h(q)=b.
\]
\end{df}

\par Any function $\partial \in \mathcal{O}(A)_{[3]}$ satisfying
$\partial(xxy)=\partial(xyx)=\partial(yxx)=x$ is called a
\emph{$\partial$-function} (the terms \emph{majority function} and
\emph{discriminator} are also often used). Any function $w \in
\mathcal{O}(A)$ satisfying $w(xx\ldots x)=x$ is called an
\emph{idempotent function}.

\begin{thm}\label{partial} Let $A$, $Q$ and $H\subseteq A^Q$ be non-empty finite sets. Let $\mathcal{F}$ be a clone with the carrier
$A$, and let at least one of the following two conditions hold:
\begin{enumerate}
\item[(a)] $\mathcal{F}$ satisfies $\Delta^{\partial}$,
\item[(b)] $\mathcal{F}$ contains a $\partial$-function, and $|H(q)|\leq
2$ for all $q\in Q$.
\end{enumerate}
\par Then $H\in\mathrm{Inv}_Q \mathcal{F}$ if and only if the
following conditions hold:
\begin{enumerate}
\item\label{partsareinv} $H|_{P}\in \mathrm{Inv}_P \mathcal{F}$ for all $P\in [Q]^1\cup [Q]^{2,0}_H\cup
[Q]^{2,1}_H$,
\item\label{decompparts} $H$ is decomposable over $[Q]^1\cup [Q]^{2,0}_{H}\cup [Q]^{2,1}_{H}$.
\end{enumerate}
\end{thm}

\begin{proof}

\par If $|Q|=1$, the theorem is obvious. Assume $|Q|\geq 2$. In the \emph{if} direction
the theorem follows immediately from Proposition \ref{InvPolprop}.
Let us prove the \emph{only if} direction. Let $H\in
\mathrm{Inv}_Q\mathcal{F}$. Item \ref{partsareinv} again follows
from Proposition \ref{InvPolprop}. To prove item \ref{decompparts},
we first prove the following Lemmas. We assume that all the premises
of the theorem hold.

\begin{lm}\label{slabsil} Let $p,q\in Q$, $a\in H(p)$ and $H$
weakly separate $p$ from $q$ at the point $a$. Then $H$ strongly
separates $p$ from $q$ at the point $a$.
\end{lm}

\begin{proof} Obviously, the lemma is true if $H(q)\leq 2$. Suppose, on the
contrary, that $H(q)\geq 3$. Let $b\in H(q)$. Suppose that $H$ does
not contain a function $h$ such that $h(p)=a$ and $h(q)=b$. Choose
functions $h_0,h_1, h_2\in H$ and distinct elements $d,c\in A$ for
which
\[
h_0(p)=h_1(p)=a,\,h_0(q)=c,\, h_1(q)=d,\, h_2(q)=b.
\]
By the above supposition, $b\notin\{\,c,d\,\}$, so $cdb\in A^3_3$.
By the premises of the theorem there is a $\partial$-function
$\partial\in \mathcal{F}$ such that $\partial(cdb)=b$. Consider the
function $f=\partial(h_0, h_1, h_2)$. Since $H\in \mathrm{Inv}_Q
\mathcal{F}$ we have $f\in H$.  However, it is easy to calculate
that $f(p)=a$ and $f(q)=b$, a contradiction.
\end{proof}

\begin{lm}\label{otdel} Let $P=\{\,p,q\,\}\in [Q]^{2}$. Then one of the following three
cases holds:
\begin{enumerate}
\item\label{otdel1} $H|_P$ is the set of all functions $h\in A^Q$ such that $h(p)\in H(p)$ and $h(q)\in H(q)$,
\item\label{otdel2} $H|_P$ is the set of all functions $h\in A^Q$ such that $h(p)\in H(p)$ and $h(q)\in
H(q)$, and $h(q)=\sigma(h(p))$ for some $\sigma\in S_A$,
\item\label{otdel3} $H|_P$ is the set of all functions $h\in A^Q$ such that $h(p)\in H(p)$ and $h(q)\in
H(q)$, and \emph{(}$h(p)=a\vee h(q)=b$\emph{)} for some $a,b\in A$.
\end{enumerate}
\end{lm}

\begin{proof} Let $H$ do not weakly separate $p$ and $q$. Since $H$ does
not weakly separates $p$ from $q$, there is a function $\sigma\colon
H(p)\to H(q)$ such that $h(q)=\sigma(h(p))$ for all $h\in H$.
Obviously, $\sigma$ is a surjective function. Suppose that
$\sigma(a)=\sigma(b)$ for some distinct $a,b\in H(p)$. Choose a
function $h_0, h_1\in H$ such that $h_0(p)=a$ and $h_1(p)=b$. We
have $h_0(q)=h_1(q)$. This means that $H$ weakly separates $q$ from
$p$ at the point $h_1(q)$, a contradiction. Therefore, $\sigma$ is a
one-to-one mapping. We can extend $\sigma$ to some permutation of
$A$. So, we have the case~\ref{otdel2}.
\par Let $H$ weakly separate one of the elements $p$, $q$ from the other at least at two distinct points. Without loss of generality, we assume
that $H$ weakly separates $p$ from $q$ at least at two distinct
points. By Lemma \ref{slabsil},  $H$ strongly separates $p$ from $q$
at least at two distinct points.  So, for any $b\in H(q)$ there are
functions $h_0, h_1\in H$ such that $h_0(q)=h_1(q)=b$ and
$h_0(p)\neq h_1(p)$. Hence, $H$ weakly (and, by  Lemma
\ref{slabsil}, strongly) separates~$q$ from~$p$ at any point $b\in
H(q)$. We have the case~\ref{otdel1}.
\par Now let the two previous assumptions do not satisfied. Without loss of generality, we assume
that $H$ weakly separates $p$ from $q$ at the unique point~$a$. If
$|H(p)|=1$, by  Lemma \ref{slabsil}, we have case~\ref{otdel1}. Let
$|H(p)|\geq 2$, $a'$ be an arbitrary element of
$H(p)\setminus\{\,a\,\}$, and $h$ an arbitrary function in $H$ for
which $h(p)=a'$. By  Lemma \ref{slabsil} there is a function $h_0\in
H$ such that $h_0(p)=a$ and $h_0(q)=b$. So, $H$ weakly separates $q$
from $p$ at the point $b$. By assumption, there is at most one such
point $b\in H(q)$. We have $h(q)=b$ for any function $h$ satisfying
$h(p)\neq a$, the case~\ref{otdel3}.
\end{proof}

\par It follows from Lemma \ref{otdel} that it suffices to prove that $H$ is decomposable over $[Q]^2$.
Define
\[
H^\ast=\bigcap_{P\in[Q]^2}(H|_P)|^Q.
\]
By Proposition \ref{decomp}, $H\subseteq H^\ast$. Therefore, it
suffices to prove the reverse inclusion.
\par We prove that for any set $Q'\subseteq Q$ and function $f\in
H^\ast$ there exists a function $h\in H$ such that
$f|_{Q'}=h|_{Q'}$. By induction on the cardinality of $Q'$. If
$|Q'|=1$, this follows from Proposition \ref{decomp}.
\par Let $|Q'|\geq 2$ and $f\in H^\ast$. Choose two distinct $p,q\in Q'$. By the induction
hypothesis there are two functions $h_p, h_q\in H$ that coincide
with  $f$ on  \mbox{$Q'\setminus\{\,p\,\}$} and \mbox{$Q'\setminus
\{\,q\,\}$}, respectively. In addition, by definition of $H^\ast$,
there is a function $h_{pq}\in H$ that coincides with  $f$ on
$\{\,p,q\,\}$, i.e. $h_{pq}(p)=f(p)$ and $h_{pq}(q)=f(q)$. Choose an
arbitrary \mbox{$\partial$-function} $\partial\in \mathcal{F}$, and
consider the function $h=\partial(f_p,f_q,f_{p,q})$. We have $h\in
H$ because $H\in \mathrm{Inv}_Q \mathcal{F}$. Moreover, the
following equalities are true:
\begin{align*}
&h(p)=\partial(f_p(p)f_q(p)f_{p,q}(p))=\partial(f_p(p)f(p)f(p))=f(p),\\
&h(q)=\partial(f_p(q)f_q(q)f_{p,q}(q))=\partial(f(q)f_q(q)f(q))=f(q),\\
&h(x)=\partial(f_p(x)f_q(x)f_{p,q}(x))=\partial(f(x)f(x)f_{p,q}(x))=f(x)
\end{align*}
for all $x\in Q'\setminus\{\,p,q\,\}$. The induction step is proved.

\end{proof}

\begin{thm}\label{s3} Let $A$, $Q$ and $H\subseteq A^Q$ be non-empty finite sets and $n$ a natural number, $n\geq 3$. Let $\mathcal{F}$ be
a clone with the carrier $A$ satisfying $\Delta^{s}_{n}$.

\par Then $H\in\mathrm{Inv}_Q \mathcal{F}$ if and only if
\begin{enumerate}
\item\label{partsareinv2} $H|_{P}\in \mathrm{Inv}_P \mathcal{F}$ for all $P\in [Q]^1\cup [Q]^{2,0}_H\cup
\{\,Q^{(n)}_H\,\}$,
\item\label{decompparts2} $H$ is decomposable over $[Q]^1\cup [Q]^{2,0}_{H}\cup \{\,Q^{(n)}_H\,\}$.
\end{enumerate}

\end{thm}

\begin{proof}

\par If $|Q|=1$, the theorem is obvious. Assume $|Q|\geq 2$. In the \emph{if} direction
the theorem follows immediately from Proposition \ref{InvPolprop}.
Let us prove  the \emph{only if} direction. Let $H\in
\mathrm{Inv}_Q\mathcal{F}$. Item \ref{partsareinv2} again follows
from Proposition \ref{InvPolprop}. To prove item \ref{decompparts2},
we first prove the following Lemmas. We assume that all the premises
of the theorem hold.

\begin{lm}\label{alli}
Let $j<n$, $\mathbf{b}\in A^{n}_{n}$ and $b\in \mathrm{ran}\,
\mathbf{b}$. Then there is a function \mbox{$t\in \mathcal{F}$} for
which $t(\mathbf{b})=b$ and $t(\mathbf{x})=x_j$ for all
$\mathbf{x}=x_0x_1\ldots x_{n-1}\in A^{n}_{<n}$.
\end{lm}

\begin{proof}
By condition $\Delta^{s}_{n}$ we have that there is a natural number
$i<n$ such that for all $\mathbf{a}\in A^{n}_{n}$ and $a\in
\mathrm{ran}\, \mathbf{a}$ there is a function $s\in
\mathcal{F}_{[n]}$ for which $s(\mathbf{a})=a$ and
$s(\mathbf{x})=x_i$ for all $\mathbf{x}=x_0x_1\ldots x_{n-1}\in
A^{n}_{<n}$. If $j=i$, the lemma is proved.  Let $j\neq i$, and
$\tau$ be the transposition $(i,j)\in S_{n}$. Choose a function
$s\in \mathcal{F}_{[n]}$ for which $s(\mathbf{b}(\tau))=b$ and
$s(\mathbf{x})=x_i$ for all $\mathbf{x}=A^{n}_{<n}$. It is easy to
see that we can put $t=s(x_{\tau(0)},x_{\tau(1)},\ldots,
x_{\tau(n-1)})$.
\end{proof}

\begin{lm}\label{slabsil2} Let $p,q\in Q$, $a\in H(p)$, $|H(q)|\geq n$ and $H$
weakly separate $p$ from $q$ at the point $a$. Then $H$ strongly
separates $p$ from $q$ at the point $a$.
\end{lm}

\begin{proof} For an arbitrary element $b\in H(q)$, assume that
$H$ does not contain a function $h$ such that $h(p)=a$ and $h(q)=b$.
Choose a functions $h_0, h_1, h_2\in H$, and distinct elements
$c,d\in A$ such that
\[
h_0(p)=h_1(p)=a,\,h_0(q)=c,\, h_1(q)=d,\, h_2(q)=b.
\]
By the above assumption, $b\notin\{\,c,d\,\}$. Using the inequality
$|H(q)|\geq n$, choose $n-3$ functions $h_3,h_4,\ldots,h_{n-1}$ from
$H$ such that $\mathbf{b}=h_0(q)h_1(q)h_2(q)\ldots h_{n-1}(q)$ is a
repetition-free sequence. The sequence
$\mathbf{a}=h_0(p)h_1(p)h_2(p)\ldots h_{n-1}(p)$ belongs
to~$A^{n}_{<n}$. By Lemma \ref{alli}, there exists a function $t\in
\mathcal{F}$ for which $t(\mathbf{a})=h_0(p)=a$ и
$t(\mathbf{b})=h_2(q)=b$. Consider the function $h=t(h_0,h_1,\ldots,
h_{n-1})$. Since $H\in \mathrm{Inv}_Q \mathcal{F}$ we have $h\in H$.
However, it is easy to calculate that $h(p)=t(\mathbf{a})=a$ и
$h(q)=t(\mathbf{b})=b$; a contradiction.
\end{proof}

\begin{lm}\label{otdele} Let $P=\{\,p,q\,\}\in [Q]^{2}$, $P\nsubseteq Q^{(n)}_H$. Then one of the following
two cases holds:
\begin{enumerate}
\item\label{otdele1} $H|_P$ is the set of all functions $h\in A^P$ such that $h(p)\in H(p)$ and $h(q)\in H(q)$,
\item\label{otdele2} $H|_P$ is the set of all functions $h\in A^P$ such that $h(p)\in H(p)$ and $h(q)\in
H(q)$, and $h(q)=\sigma(h(p))$ for some $\sigma\in S_A$.
\end{enumerate}
\end{lm}

\begin{proof} If $H$ does not weakly separate $p$ and $q$ then we have the
case \ref{otdele2}, which can be proved in the same way as in
Theorem \ref{partial}.
\par Otherwise, we show that the case \ref{otdele1} holds.
Without loss of generality, we assume that $|H(p)|\leq |H(q)|$.
Therefore, we have $|H(q)|\geq n$. We show that without loss of
generality we can assume that $H$ weakly separates $p$ from $q$. In
fact, otherwise there exists a surjective function $\sigma\colon
H(p)\to H(q)$. Hence, $|H(p)|=|H(q)|$, and $p$ and~$q$ can be
interchanged if necessary.
\par Let $H$ weakly separate $p$ from $q$ at the point $a\in
H(p)$. By Lemma~\ref{slabsil2}, it suffices to show that $H$ weakly
separates $p$ from $q$ at each point $a'\in H(p)\setminus\{\,a\,\}$.
\par Choose an arbitrary element $a'\in H(p)\setminus\{\,a\,\}$. Using Lemma \ref{slabsil2}, choose
some functions $h_0,h_1,\ldots,h_{n-1}\in H$ such that $h_0(p)=a'$,
$h_1(p)=h_2(p)=\ldots=h_{n-1}(p)=a$, and
$\mathbf{b}=h_0(q)h_1(q)h_2(q)\ldots h_{n-1}(q)$ is a
repetition-free sequence. Since the sequence
$\mathbf{a}=h_0(p)h_1(p)h_2(p)\ldots h_{n-1}(p)=a'a\ldots a$ belongs
to $A^{n}_{< n}$, it follows from Lemma \ref{alli} that
$\mathcal{F}$ contains a function $t$ such that $t(\mathbf{a})=a'$
and $t(\mathbf{b})\neq h_{0}(q)$. Then the values of the functions
$h_0$ and $h= t(h_0,h_1,h_2,\ldots, h_{n-1})$ coincide (and are
equal to $a'$) on $p$ and are different on $q$. The function $h$
belongs to $H$ because $H\in \mathrm{Inv}_Q \mathcal{F}$. So, $H$
weakly separates $p$ from $q$ at the point $a'$.
\end{proof}

\par It follows from Lemma \ref{otdele} that it suffices to prove that $H$ is decomposable over $[Q]^2\cup \{\,Q^{(n)}_H\,\}$.
Define
\[
H^\ast=(H|_{Q^{(n)}_H})|^Q\cap \bigcap_{P\in[Q]^2}(H|_P)|^Q.
\]
By Proposition \ref{decomp}, $H\subseteq H^\ast$. Therefore, to
prove the second part of the theorem it suffices to prove the
reverse inclusion.

\par We prove that for any set $Q'\subseteq Q$ and function $f\in
H^\ast$ there exists a function $h\in H$ such that
$f|_{Q'}=h|_{Q'}$. By induction on the cardinality of $Q'$. If
$Q'\subseteq Q^{(n)}_H$ or $|Q'|=1$, the statement follows from
Proposition \ref{decomp}.

\par Let $|Q'|\geq 3$, $Q'\nsubseteq Q^{(n)}_H$, and $f\in H^\ast$. Choose two distinct $p,q\in Q'$ such that $H(q)\geq n$. By the induction
hypothesis there are two functions $h_p, h_q\in H$ that coincide
with  $f$ on  $Q'\setminus\{\,p\,\}$ and $Q'\setminus \{\,q\,\}$,
respectively. In particulary,
\[
h_q(p)=f(p)\,\,\text{and}\,\,h_p(q)=f(q).
\]
\par If $h_{q}(q)=f(q)$, we put $h=f_{q}$.
\par Let $f_{q}(q)\neq h(q)$. Consider the set $H|_{P}$ where $P=\{\,p,q\,\}$. If the case \ref{otdele2} of the Theorem holds, we have
\[
h_{q}(q)=\sigma (h_q(p))=\sigma (f(p))=f(q),
\]
a contradiction.
\par Then the case \ref{otdele1} of the Theorem holds.  Choose $n-2$ functions $h_2$, $h_3$, $\ldots$,
$h_{n-1}$ from $H$ for which $h_2(p)=h_3(p)=\ldots=h_{n-1}(p)=f(p)$
and
\[
\mathbf{b}=h_q(q) h_p(q)h_2(q)h_3(q)\ldots h_{n-1}(q)
\]
is a repetition-free sequence. Using Lemma~\ref{alli} choose a
function $t\in \mathcal{F}$ such that $t(\mathbf{b})=h_p(q)=f(q)$
and $t(\mathbf{x})=x_0$ for all $\mathbf{x}=x_0x_1\ldots x_{n-1}\in
A^{n}_{<n}$. Consider the function $h=t(h_q, h_p, h_2, h_3,\ldots,
h_{n-1})$. We have $h\in H$ because $H\in \mathrm{Inv}_Q
\mathcal{F}$. Note that the sequence
\[
\mathbf{a}=h_q(p) h_p(p)h_2(p)h_3(p)\ldots h_{n-1}(p)
\]
belongs to $A^{n}_{<n}$. Moreover, for all $x\in
Q'\setminus\{\,p,q\,\}$ we have $h_p(x)=h_{q}(x)=f(x)$, and,
therefore,  the sequence
\[
h_q(x)h_p(x)h_2(x)h_3(x)\ldots h_{n-1}(x)
\]
also belongs to $A^{n}_{<n}$. Consequently,
\begin{align*}
&h(p)=t(\mathbf{a})=h_q(p)=f(p),\\
&h(q)=t(\mathbf{b})=h_p(q)=f(q),\\
&h(x)=t(h_q(x)h_p(x)h_2(x)h_3(x)\ldots h_{n-1}(x))=h_q(x)=f(x)
\end{align*}
for all $x\in Q'\setminus\{\,p,q\,\}$. The induction step is proved.
\end{proof}

\begin{thm}\label{d2} Let $A$, $Q$ and $H\subseteq A^Q$ be non-empty finite sets. Let $\mathcal{F}$ be a clone with the carrier $A$ satisfying
$\Delta^{2}$.
\par Then $H\in\mathrm{Inv}_Q \mathcal{F}$ if and only if
\begin{enumerate}
\item\label{partsareinv3} $H|_{P}\in \mathrm{Inv}_P \mathcal{F}$ for all $P\in [Q]^1\cup
\{\,Q^{[B]}_H\colon  B\in [A]^2\,\}$,
\item\label{decompparts3} $H$ is decomposable over $[Q]^1\cup [Q]^{2,id}_{H}\cup \{\,Q^{[B]}_H\colon  B\in [A]^2\,\}$.
\end{enumerate}
\end{thm}

\begin{proof} As in the previous theorems, in the direction \emph{if}
the theorem follows immediately from Proposition \ref{InvPolprop}
(note that for any $P\in [Q]^{2,id}_{H}$ the set $H|_{P}$ is
preserved by any function $f\in \mathcal{O}(A)$). Let us prove the
theorem in the opposite direction. Let $H\in
\mathrm{Inv}_Q\mathcal{F}$. Item \ref{partsareinv3} again follows
from Proposition \ref{InvPolprop}.
\par Define
\[
H^\ast= \bigcap_{B\in [A]^2}(H|_{Q^{[B]}_H})|^{Q}\cap \bigcap_{P\in
[Q]^{2, id}_{H}} (H|_{P})|^{Q}\cap \bigcap_{q\in Q}
(H|_{\{\,q\,\}})|^{Q}.
\]
By Proposition \ref{decomp}, $H\subseteq H^\ast$. Therefore, to
prove the theorem it suffices to prove the reverse inclusion.
\par We prove that for any set $Q'\subseteq Q$ and function $f\in
H^\ast$ there exists a function $h\in H$ such that
$f|_{Q'}=h|_{Q'}$. Induction on the cardinality of $Q'$. If $|Q'|=1$
or $Q'\subseteq Q^{[B]}_H$ for some $B\in [A]^2$, the statement
follows from Proposition \ref{decomp}.
\par Let $|Q'|\geq 2$, $Q'\nsubseteq Q^{[B]}_H$ for all $B\in [A]^2$, and $f\in
H^\ast$. By the induction hypothesis, for any $q\in Q'$ there is a
function $h_{q}\in H$ which coincides with $f$ on
$Q'\setminus\{\,q\,\}$. Fixe some family $\{\,h_q\,\}_{q\in Q'}$ of
such functions.

\begin{lm}\label{2razn} At least one of the following conditions
holds:
\begin{enumerate}
\item\label{2razn1} There is $q\in Q'$ for which $h_q(q)=f(q)$.
\item\label{2razn2} $[Q]^{2,id}_{H}\cap[Q']^2= \emptyset$, and for all $q\in Q'$ and $a\in H(q)$
there is a function $f_{q,a}\in H^\ast$ satisfying
$$
f_{q,a}|_{Q'\setminus\{\,q\,\}}=h_q|_{Q'\setminus\{\,q\,\}}\,\text{and}\,\,f_{q,a}(q)=a.
$$
\end{enumerate}
\end{lm}

\begin{proof}
\par Let $q\in Q'$, $a\in H(q)$ and $f_{q,a}^-$ be a function from $Q'$ to $A$, satisfying
\[
f_{q,a}^-=h_q|_{Q'\setminus\{\,q\,\}}\,\text{and}\,\,f_{q,a}^-(q)=a.
\]

\par Let condition \ref{2razn1} do not hold. If
$[Q]^{2,id}_{H}\cap[Q']^2\neq \emptyset$, we have
$h_{q'}(q')=h_{q'}(p')=f(p')=f(q')$ for some distinct $p',q'\in Q'$,
a contradiction.

\par Hence, it suffices to show that $f_{q,a}^-$ belongs to each of
the sets $H|_{Q^{[B]}_H\cap Q'}$ where $B\in [A]^2$.
\par Suppose that there is $B\in [A]^2$ such that
\[
f_{q,a}^-\notin H|_{Q^{[B]}_H\cap Q'}.\eqno{(\ast)}
\]
This means that for all $g\in H$, if $g$ and $f_{q,a}^-$ coincide on
the set $(Q^{[B]}_H\cap Q')\setminus \{\,q\,\}$, then $g(q)\neq a$.
Recall that $h_q$ and $f_{q,a}^-$ coincide on
$Q'\setminus\{\,q\,\}$. Therefore, \mbox{$h_q(q)\neq a$}.
\par On the other hand, $f$ and $f_{q,a}^-$ coincide
on $Q'\setminus\{\,q\,\}$. Since $f\in H^\ast$, we have
\[
f|_{Q^{[B]}_H}\in H|_{Q^{[B]}_H}.
\]
If $q\notin Q^{[B]}_H$ or $f(q)=f_{q,a}^-(q)=a$ these conditions
contradict supposition~$(\ast)$. So, $q\in Q^{[B]}_H$, and,
consequently, $|H(q)|\leq 2$. Moreover, $f(q)\neq a$. Therefore, we
have $h_q(q)=f(q)$, a contradiction.
\end{proof}

\par We continue the proof of the induction step. If $h_q(q)=f(q)$ for some $q\in Q'$ we can put $h=h_q$. In the
sequel, we assume that
\[
h_q(q)\neq f(q)\eqno{(\ast\ast)}
\]
for all $q\in Q'$. Then by Lemma \ref{2razn} and the induction
hypothesis, for any distinct $p,q\in Q'$ and $a\in H(q)$ the set $H$
contains a function $h_{p,q,a}$ that coincides with $f_{q,a}$ on
$Q'\setminus \{\,p\,\}$. Thus, for all distinct $p,q\in Q'$ and
$a\in H(q)$ we have
\begin{align*}
&h_{p,q,a}|_{Q'\setminus \{\,p,q\,\}}=h_q|_{Q'\setminus
\{\,p,q\,\}}=f|_{Q'\setminus \{\,p,q\,\}}\\
&h_p(q)=f(q), h_q(p)=f(p)\\
&h_{p,q,a}(q)=a\\
&h_p, h_q, h_{p,q,a}\in H.
\end{align*}
\par \textbf{Case 1:} $\mathrm{ran}\, f(p)h_{p}(p)\neq \mathrm{ran}\, h_{q}(q)f(q)$ for some distinct $p, q\in Q'$.
\par Choose such $p$ and $q$. Using $\Delta^2$ choose an idempotent function $w_0\in
\mathcal{F}_{[2]}$ such that
\[
w_0(f(p) h_{p}(p))= f(p)\,\, \text{and}\,\, w_0(h_{q}(q) f(q))=
f(q).
\]
\par We can put \mbox{$h=w_0(h_{q},h_{p})$}.

\par \textbf{Case 2:} $H(p)\neq H(q)$ for some $p, q\in Q'$.
\par Choose such $p$ and $q$. Without loss of generality, assume $H(p)\setminus
H(q)\neq \emptyset$. Choose \mbox{$a\in H(p)\setminus H(q)$}.  Using
$\Delta^2$ choose idempotent functions $w_0,w_1\in
\mathcal{F}_{[2]}$ such that
\[
\begin{array}{ll}
  w_0(f(p)a)=f(p), & w_0(h_{q}(q)f(q))=f(q) \\
  w_1(h_p(p)a)=a, & w_1(f(q)h_{q,p,a}(q))=f(q). \\
\end{array}
\]
\par We can put $h=w_0(h_{q}, w_1(h_{p},h_{q,p,a}))$.

\par \textbf{Case 3:} Case 1 and Case 2 do not hold.
\par We have $H(p)=H(q)$ for all $p,q\in Q'$. Denote $C=H(p)$ for
some $p\in Q'$. If $|C|\leq 2$ we go back to the induction base.
Further, we assume that $|C|\geq 3$.
\par \textbf{Subcase 3.1:} there exist distinct $p,q\in Q'$ for which
$f(p)=f(q)$. Choose such $p$ and $q$ and denote $f(p)=a$. Since the
assumption $(\ast\ast)$ is true and the Case 1 does not hold, we
have $h_{p}(p)=h_{q}(q)=b$ for some $b\in A\setminus\{\,a\,\}$.
Choose $c\in C\setminus \{\,a,b\,\}$.  Using $\Delta^2$ choose
idempotent functions $w_0,w_1\in \mathcal{F}_{[2]}$ such that
\[
w_0(ac)=w_0(ba)=a,\,
w_1(bc)=c\,\,\text{and}\,\,w_1(af_{q,p,c}(q))=a.
\]
\par We can put $h=w_0(h_{q},w_1(h_{p},h_{q,p,c}))$.
\par \textbf{Subcase 3.2:} $f(p)\neq f(q)$ for all distinct $p,q\in
Q'$.
\par First, let $Q'$ contain only two elements $p$ and $q$. Denote
$f(p)=a$ and $f(q)=b$. Since the assumption $(\ast\ast)$ is true and
the Case 1 does not hold, we have $f_{q}(q)=a$ and $f_{p}(p)=b$.
\par Further, by the assumption $(\ast\ast)$ and Lemma \ref{2razn} we have that there is a function $h'\in H$ such that $h'(p)\neq
h'(q)$. Denote $h'(p)=c$ and $h'(q)=d$.
\par We show that we can choose $c\neq b$. Consider the set
\[
D=\{\,x\in C\colon \text{for all $h\in H$}\,\,h(p)=x\rightarrow
h(q)=x\,\}.
\]
It suffices to prove $|D|\leq 1$. Let $x\in D\setminus \{\,d\,\}$.
Choose a function $h''\in H$ such that $h''(p)=x$.  Using $\Delta^2$
choose an idempotent functions $w\in \mathcal{F}_{[2]}$ for which
$w(cx)=x$ and $w(dx)=d$. Then we have $h'''=w(h',h'')\in H$,
$h'''(p)=x$ and $h'''(q)=d$, a contradiction. So, $D\subseteq
\{\,d\,\}$.
\par Now using $\Delta^2$ choose idempotent
functions $w_0,w_1\in \mathcal{F}_{[2]}$ such that
\[
w_{0}(ac)=a,\,w_{0}(ab)=b,\, w_1(bc)=c,\, w_1(bd)=b.
\]
\par We can put $h=w_0(h_{q},w_1(h_{p},h'))$.
\par Now let $|Q'|\geq 3$. Choose pairwise different $p,q,r\in Q'$. Denote $f(p)=a$, $f(q)=b$ и $f(r)=c$.
By by the assumption $(\ast\ast)$ and Lemma \ref{2razn} the function
$f_{p,b}$ and $f_{q,c}$ belong to $H^\ast$. In addition, each of
them satisfies subcase 3.1. Consequently, there are functions
$h^{\ast}_{p,b}, h^{\ast}_{q,c}\in H$ which coincide with the
functions $f_{p,b}$ and $f_{q,c}$ on the whole set $Q'$,
respectively. Using $\Delta^2$ choose an idempotent function $w_0\in
\mathcal{F}_{[2]}$ for which
\[
w_{0}(ba)=a,\, w_{0}(bc)=b.
\]
\par We can put $h=w_{0}(h^{\ast}_{p,b},h^{\ast}_{q,c})$. The induction step is proved.
\end{proof}

\par\textbf{Remarks} Theorem \ref{partial} is close to the main result of paper \cite{Marchen1997}.
In addition, Theorems \ref{partial} and \ref{s3} in a particular
case are proved in \cite{Shelah2005}. In our proofs, we used some
ideas of \cite{Shelah2005}.

\section{Conservative clones}

In this section, we apply decomposition theorems to the
classification of symmetric conservative clones with a finite
carrier $A$.
\begin{df}
A function $f\in \mathcal{O}(A)$ is called \emph{conservative} if
$$
f(\textbf{a})\in \mathrm{ran}\, \mathbf{a}
$$
for each $\mathbf{a}\in \mathrm{dom}\, f$. .
\end{df}

Note that in terms of $(\mathrm{Inv}, \mathrm{Pol})$-connection a
conservative function can be characterized as a function preserving
any unary predicate.

\par It easy to check that the set of all conservative function $f\in
\mathcal{O}(A)$ is a clone. We denote this clone by the symbol
$\mathcal{C}(A)$. Any clone $\mathcal{F}\subseteq \mathcal{C}(A)$ is
called \emph{conservative}. The case $|A|=1$ is of no interest.
Therefore, we further believe that $|A|\geq 2$.

\par Now we introduce the concept of \emph{characteristic} of
conservative clone.

\begin{df}
Let $\mathcal{F}$ be a clone with a carrier $A$. Then
$\mathrm{r}(\mathcal{F})$ is the minimal natural number $r$ for
which there exists a $r$-ary function $f\in \mathcal{F}$ that is not
a projection. If $\mathcal{F}=\mathcal{E}(A)$ we put
$\mathrm{r}(\mathcal{F})=\omega$.
\end{df}

\par Obviously, if a clone $\mathcal{F}$ is conservative, then it does not
contain $0$-ary functions and any unary function $f\in\mathcal{F}$
is a projection. So, $\mathrm{r}(\mathcal{F})\geq 2$ for each
conservative clone $\mathcal{F}$.

\begin{df}
\emph{Let $\mathcal{F}$ be a clone with a carrier $A$. Then, for any
natural number $n$, $\mathrm{R}_n(\mathcal{F})$ is a binary relation
on~$A^n$ defined by
$$
\mathbf{a}\,\mathrm{R}_n(\mathcal{F})\,\mathbf{b} \leftrightarrow
(\exists \sigma \in S_A)(\forall f\in
\mathcal{F}_{[n]})\,f(\mathbf{b})=\sigma f(\mathbf{a}).
$$
We define
$\mathrm{R}(\mathcal{F})=\bigcup_{n<\omega}\mathrm{R}_n(\mathcal{F}).$}
\end{df}

Therefore, $\mathrm{R}(\mathcal{F})$ is a binary relation on
$A^{<\omega}$.

\begin{df}
Let $\mathcal{F}$ be a clone with a carrier $A$. Then, for any
natural number $n$, $\mathrm{D}_n(\mathcal{F})$ is a binary relation
on~$A^n$ defined by
\[
\mathbf{a}\,\mathrm{D}_n(\mathcal{F})\,\mathbf{b} \leftrightarrow
(\exists a,b\in a)(\forall f\in \mathcal{F}_{[n]})\,f(\mathbf{a})=
a\vee f(\mathbf{b})=b.
\]
We define
$\mathrm{D}(\mathcal{F})=\bigcup_{n<\omega}\mathrm{D}_n(\mathcal{F}).$
\end{df}

Therefore, $\mathrm{D}(\mathcal{F})$ is a binary relation on
$A^{<\omega}$.

\par Now let $\mathcal{F}$ is a conservative clone with a carrier $A$.
It easy to check that for any set $B\in [A]^2$ the set
$\mathcal{F}|_{B^{<\omega}}$ is a clone with the carrier $B$, and
there exists a natural isomorphism $(\sigma_B, \tau_B)$ from
$\mathcal{F}|_{B^{<\omega}}$ to some Post's class~$\Pi_B$ of Boolean
functions preserving $\mathbf{0}$ and $\mathbf{1}$ (A Boolean
function $f$ \emph{preserves $\mathbf{0}$ and $\mathbf{1}$} if
$f(0,0,\ldots, 0)=0$ and $f(1,1,\ldots,1)=1$, i.e. if $f$ is
conservative). Note that $\Pi_B$ is defined up to a natural
isomorphism of Post's classes. For definiteness, we choose some
maximal by inclusion set $\mathbb{P}$ of pairwise non natural
isomorphic Post's classes and assume that~$\Pi_B$ belongs to
$\mathbb{P}$ for any $B\in [A]^2$. In this case the class $\Pi_B$ is
uniquely determined for each $B\in [A]^2$. Also, in this case the
natural isomorphism $(\sigma_B, \tau_B)$ is uniquely determined if
$\Pi_B$ is not closed with respect to duality (Post's class $P$ is
closed with respect to duality if for all natural number $n$ and
$n$-ary function $f\in P$ the function
$f^\ast=\overline{f(\overline{x}_0, \overline{x}_1, \ldots,
\overline{x}_{n-1})}$ belongs to $P$, i.e. if $P$ is symmetric). If
$\Pi_B$ is closed with respect to duality, there are two distinct
natural isomorphisms from $\mathcal{F}|_{B^{<\omega}}$ to $\Pi_B$.
For definiteness, we assume that $(\sigma_B, \tau_B)$ is any one of
these two natural isomorphisms.

\begin{df} Let $\mathcal{F}$ be a conservative clone with a carrier
$A$. The family of natural isomorphisms $\{\,(\sigma_B,
\tau_B)\,\}_{B\in [A]^2}$  from $\mathcal{F}|_{B^{<\omega}}$ to
$\Pi_B$ is denoted by~$\Pi(\mathcal{F})$.
\end{df}

\begin{df}The quadruple $\chi(\mathcal{F})=(\mathrm{r}(\mathcal{F}),
\mathrm{R}(\mathcal{F}), \mathrm{D}(\mathcal{F}), \Pi(\mathcal{F}))$
is called the \emph{characteristic} of a conservative clone
$\mathcal{F}$.
\end{df}

\par We show that the characteristic $\chi(\mathcal{F})$ uniquely determines a conservative symmetric
clone $\mathcal{F}$ with a finite carrier. In addition to
decomposition theorems, Post's classification of closed classes of
Boolean functions is used in the following proofs. Post's
classification can be found in \cite{Post} or (in a more modern
version) in \cite{Lau}. We will not refer to these works every time
we use them. We need only the classification of closed classes of
Boolean functions that are closed with respect to duality and
consist of functions preserving $\mathbf{0}$ and $\mathbf{1}$.
Recall that there are only six such classes: $O_1$, $D_1$, $D_2$,
$L_4$, $A_4$ and $C_4$ (in Post's notation). Of these, the classes
$O_1$, $D_1$, $D_2$, and  $L_4$ consist of self-dual functions. They
are generated by the functions $x$,
$\overline{x}y\vee\overline{x}z\vee yz$, $xy\vee yz\vee xz$ и
$x\oplus y \oplus z$, respectively. Moreover, $L_4\cup D_2\subseteq
D_1$.
\par We begin with the following lemma, which concerns not only symmetric clones.

\begin{lm}\label{rparametr} Let $|A|\geq 2$. Let $\mathcal{F}\subseteq \mathcal{O}(A)$ be a conservative clone, and $\mathrm{r}(\mathcal{F})=r\geq 3$.
Then there is Post's class $P\in \{\,O_1, D_1, D_2, L_4 \,\}$ and a
surjective mapping $\tau\colon \mathcal{F}\to P$ such that for each
set $B\in [A]^2$, one-to-one mapping $\sigma\colon B\to 2$, natural
number $n$, function $f\in \mathcal{F}_{[n]}$ and $n$-tuple
$\mathbf{a}\in B^n$,
$$
f(\mathbf{a})=\sigma^{-1}(\tau(f)(\sigma(\mathbf{a}))).
$$
 If
$r\geq 4$ then $P=O_1$ and  any $n$-ary function $f\in \mathcal{F}$
coincides with some projection on a set $A^{n}_{<r}$.
\end{lm}

\begin{proof} The proof follows easily from the following claim.

\begin{claim}\label{sigma} Let $\mathcal{F}\subseteq \mathcal{O}(A)$ be a conservative clone, and $\mathrm{r}(\mathcal{F})=r\geq 3$.
Let $\mathbf{a}=a_0a_1\ldots a_{n-1}\in A^{n}_{<r}$, $f\in
\mathcal{F}_{[n]}$, and $\sigma\colon A\to A$.
\par Then $f(\sigma\cdot \mathbf{a})=\sigma (f(\mathbf{a}))$.
\end{claim}

\begin{proof} It is easy to see that $f(\sigma(\mathbf{a}))=\sigma (f(\mathbf{a}))$ for any projection $f\in \mathcal{E}_{[n]}(A)$.
Further, let $\mathbf{b}=(b_0, b_1, \ldots, b_{t-1})$ be some
repetition-free sequence of all elements from $\mathrm{ran}\,
\mathbf{a}$. Denote $\xi=\mathbf{b}^{-1}(\mathbf{a})$ and consider
the function $f'=f(x_{\xi(0)},x_{\xi(1)},\ldots, x_{\xi(n-1)})\in
\mathcal{F}_{[t]}$. For any $\mathbf{c}\in A^t$ we have
\[
f'(\mathbf{c})=f(\mathbf{c}(\xi)).
\]
Since $t<r$, we have that $f'$ is a projection. Therefore,
\begin{align*} \sigma (f(\mathbf{a}))=\sigma
(f'(\mathbf{b}))=f'(\sigma(\mathbf{b}))=f(\sigma(
\mathbf{b}(\xi)))=f(\sigma(\mathbf{a})).
\end{align*}
\end{proof}

\par Without loss of generality, assume $2=\{\,0,1\,\}\subseteq A$. So, $\mathcal{F}|_{2^{<\omega}}$
is some of Post's classes of Boolean functions preserving
$\mathbf{0}$ and $\mathbf{1}$. Denote
$P=\mathcal{F}|_{2^{<\omega}}$. Claim \ref{sigma} implies that any
$g\in P$ is self-dual function (should be considered a function
$\sigma\colon A\to A$ such that $\sigma(0)=1$ and $\sigma(1)=0$).
Consequently, $P\in \{\,\mathrm{O}_1, \mathrm{D}_1, \mathrm{D}_2,
\mathrm{L}_4 \,\}$.
\par Define $\tau (f)=f|_{2^{<\omega}}$. Then $\tau$ is a surjective
function from $\mathcal{F}$ to $P$. Claim \ref{sigma} implies that
for any $B\in [A]^2$, one-to-one mapping $\sigma\colon B\to 2$,
natural number $n$, function $f\in \mathcal{F}_{[n]}$ and $n$-tuple
$\mathbf{a}\in B^n$, we have
\[
f(\mathbf{a})=\sigma^{-1}(\tau(f)(\sigma(\mathbf{a}))).
\]

\par Now let $r\geq 4$. Any class $P\in\{\,D_1, D_2, L_4 \,\}$ contains some ternary
function which is not a projection. Hence, $P=O_1$, i.e. $P$
consists only of projections $e^{n}_{i}$ ($1\leq n<\omega$, $0\leq
i<n$). Let $f$ be an arbitrary function from $\mathcal{F}_{[n]}$ for
some natural number $n$. Let~$i$ be the number for which
$\tau(f)=e^{n}_{i}$. For any $\mathbf{a}=(a_0, a_1,\ldots,
a_{n-1})\in A^{n}_{<r}$ choose a function $\sigma_{\mathbf{a}}\colon
A\to A $ for which $\sigma_{\mathbf{a}} (a_i)=1$ and
$\sigma_{\mathbf{a}}(x)=0$ for all $x\in A\setminus\{\,a_i\,\}$. By
Claim \ref{sigma} we have
\[
\sigma_{\mathbf{a}}(f(\mathbf{a}))=f(\sigma_{\mathbf{a}}(
\mathbf{a}))=(\tau(f))(\sigma_{\mathbf{a}}(\mathbf{a}))=1,
\]
whence $f(\mathbf{a})=a_i$. Therefore, $f$ coincides with the $i$-th
projection on $A^{n}_{<r}$.
\end{proof}

\par It follows from Lemma \ref{rparametr} that for any conservative clone $\mathcal{F}$ with \mbox{$\mathrm{r}(\mathcal{F})\geq
3$} each natural isomorphism $(\sigma_B,\tau_B)$, $B\in [A]^2$, maps
the clone $\mathcal{F}|_{B^{<\omega}}$ to the same Post's class $P$.
Obviously, this statement is also true for any symmetric
conservative clone. We will denote this Post's class $P$ by the
symbol $\Pi_0(\mathcal{F})$. It is easy to see that for any
symmetric conservative clone $\mathcal{F}$ the class
$\Pi_0(\mathcal{F})$ is closed with respect to duality. Thus, the
case $\mathrm{r}(\mathcal{F})= 2$ adds two more possibilities:
$\Pi_0(\mathcal{F})=A_4$ and $\Pi_0(\mathcal{F})=C_4$.

\par For any set $A$ of cardinality $3$ we denote
\[
R_{\uparrow}(A)= \{\,(ab,cd)\in A^2_2\times A^2_2\colon
ab=cd\vee(b=c\wedge a\neq d)\vee (a=d\wedge b\neq c)\,\}.
\]

\par For any set $A$ of cardinality $4$ we denote
\[
R_{\pm}(A)= \{\,(\mathbf{x},\mathbf{y})\in A^2_2\times A^2_2\colon
\mathrm{ran}\,\mathbf{x}=\mathrm{ran}\,\mathbf{y}\,\text{or}\,\mathrm{ran}\,
\mathbf{x}\cap\mathrm{ran}\,\mathbf{y}=\emptyset \,\}.
\]

\par For any set $A$ of cardinality $4$ and Post's class $P$ we say that a function $f\in \mathcal{O}(A)_{[n]}$ is Klein's
$P$-function if
\begin{enumerate}
\item for any $B\in [A]^2$, $f|_{B^{n}}$ belongs to the clone
$\mathcal{F}\subseteq \mathcal{O}(B)$ which is naturally isomorphic
to $P$,
\item for any $\mathbf{a}\in A^n_2$ and permutation $\sigma$ from Klein four-group of permutations of $A$, $\sigma
(f(\mathbf{a}))=f(\sigma(\mathbf{a}))$.
\end{enumerate}

\begin{thm}\label{main} Let $A$ be a finite set, $|A|\geq 2$. Then any symmetric
conservative clone $\mathcal{F}\subseteq \mathcal{O}(A)$ is uniquely
defined by by its characteristics, i.e.
$$
\chi(\mathcal{F})=\chi(\mathcal{G})\Rightarrow
\mathcal{F}=\mathcal{G}\eqno{(\ast\ast\ast)}
$$
for all symmetric conservative clones $\mathcal{F},
\mathcal{G}\subseteq \mathcal{O}(A)$.
\par Let $Q$ and $H\subseteq A^Q$ be non-empty finite sets and $\mathcal{F}\subseteq \mathcal{O}(A)$ a symmetric
conservative clone. Then $\mathrm{r}(\mathcal{F})\geq 2$, and
\begin{enumerate}
\item\label{rgeq4} If $\mathrm{r}(\mathcal{F})\geq 4$, then $H\in \mathrm{Inv}_Q
\mathcal{F}$ if and only if
\begin{enumerate}
\item $H|_{P}\in \mathrm{Inv}_P \mathcal{F}$ for all $P\in [Q]^{2,0}_H$,
\item $H$ is decomposable over $[Q]^1\cup [Q]^{2,0}_{H}\cup \{\,Q^{(r)}_H\,\}$.
\end{enumerate}
\item\label{req3} If $\mathrm{r}(\mathcal{F})=3$ and $\Pi_0(\mathcal{F})\neq L_4$,
then $H\in \mathrm{Inv}_Q \mathcal{F}$ if and only if
\begin{enumerate}
\item $H|_{P}\in \mathrm{Inv}_P \mathcal{F}$ for all $P\in [Q]^{2,0}_H\cup [Q]^{2,1}_H$,
\item $H$ is decomposable over $[Q]^1\cup [Q]^{2,0}_{H}\cup [Q]^{2,1}_H$.
\end{enumerate}
\item\label{req3L4} If $\mathrm{r}(\mathcal{F})=3$ and $\Pi_0(\mathcal{F})= L_4$, then $H\in \mathrm{Inv}_Q
\mathcal{F}$ if and only if
\begin{enumerate}
\item $H|_{P}\in \mathrm{Inv}_P \mathcal{F}$ for all $P\in [Q]^{2,0}_H$,
\item $H|_{Q^{(3)}_H}$ is preserved by a function $\ell\in
\mathcal{O}(A)_{[3]}$ satisfying
$\ell(x,y,y)=\ell(y,x,y)=\ell(y,y,x)=x$,
\item $H$ is decomposable over $[Q]^1\cup [Q]^{2,0}_{H}\cup \{\,Q^{(3)}_H\,\}$.
\end{enumerate}
\item\label{req2notspecial} If $\mathrm{r}(\mathcal{F})=2$ and the following cases do not
hold: (i) $|A|=4$ and \mbox{$\mathrm{R}_2(\mathcal{F})=R_{\pm}$},
(ii) $|A|=3$ and \mbox{$\mathrm{R}_2(\mathcal{F})=R_{\uparrow}$},
then $H\in \mathrm{Inv}_Q \mathcal{F}$ if and only if
\begin{enumerate}
\item $H|_{P}\in \mathrm{Inv}_P \mathcal{F}|_{B^{<\omega}}$ for all $P\in Q^{[B]}_H$ and $B\in [A]^2$,
\item $H$ is decomposable over $[Q]^1\cup [Q]^{2,id}_{H}\cup \{\,Q^{[B]}_H\colon  B\in [A]^2\,\}$.
\end{enumerate}
\item\label{req2A4} If $\mathrm{r}(\mathcal{F})=2$ and $|A|=4$, and $\mathrm{R}_2(\mathcal{F})=R_{\pm}(A)$, then $H\in
\mathrm{Inv}_Q \mathcal{F}$ if and only if
\begin{enumerate}
\item $H|_{P}\in \mathrm{Inv}_P \mathcal{F}$ for all $P\in [Q]^{2,0}_{H}$,
\item $H|_{Q^{(3)}_H}$ is preserved by all Klein's
$\Pi_0(\mathcal{F})$-functions,
\item $H$ is decomposable over $[Q]^1\cup [Q]^{2,0}_{H}\cup \{\,Q^{(3)}_H\,\}$.
\end{enumerate}
\item\label{req2A3} If $r=2$ and $|A|=3$, and $\mathrm{R}_2(\mathcal{F})=R_{\uparrow}(A)$,
then $H\in \mathrm{Inv}_Q \mathcal{F}$ if and only~if
\begin{enumerate}
\item $H|_{P}\in \mathrm{Inv}_P \mathcal{F}$ for all $P\in [Q]^{2,0}_H\cup [Q]^{2,1}_H$,
\item $H$ is decomposable over $[Q]^1\cup [Q]^{2,0}_{H}\cup [Q]^{2,1}_H$.
\end{enumerate}
\end{enumerate}
\end{thm}

\begin{proof} The proof is based on decomposition theorems and the following Lemma.
\begin{lm}\label{RandSim} Let $|A|$ be a set of cardinality $|A|\geq 2$, and $\mathcal{F}$ be a symmetric conservative clone with the carrier $A$.
Let $\mathrm{r}(\mathcal{F})=r<\omega$. Then
\begin{enumerate}
\item If $r\geq 4$, then $\mathcal F$ satisfies $\Delta^{s}_{r}$;
\item If $\mathrm{r}(\mathcal{F})=3$ and $\Pi_0(\mathcal{F})\neq L_4$, then $\mathcal F$ satisfies
$\Delta^{\partial}$;
\item If $\mathrm{r}(\mathcal{F})=3$ and $\Pi_0(\mathcal{F})= L_4$, then $\mathcal F$ satisfies
$\Delta^{s}_{3}$;
\item If $r= 2$, then one of the following cases holds:
\begin{enumerate}
\item $\mathcal F$ satisfies $\Delta^{2}$,
\item $|A|=4$ and $\mathrm{R}_2(\mathcal{F})=R_{\pm}(A)$, and $\mathcal F$
satisfies $\Delta^{s}_{3}$, and for all $B\in [A]^3$,
$\mathcal{F}|_{B^{<\omega}}$ satisfies $\Delta^2$,
\item $|A|=3$ and $\mathrm{R}_2(\mathcal{F})=R_{\uparrow}(A)$, and $\mathcal
F$ satisfies $\Delta^{\partial}$.
\end{enumerate}
\end{enumerate}
\end{lm}

\begin{proof} First we prove the following claim. All function $\ell\in \mathcal{O}(A)_{[3]}$ satisfying
\[
\ell(x,y,y)=\ell(y,x,y)=\ell(y,y,x)=x
\]
is called an $\ell$-function.

\begin{claim}\label{simDelty} Let $\mathcal{F}\subseteq \mathcal{O}(A)$ be a conservative symmetric clone and $n$ a natural number.
\begin{enumerate}
\item If $\mathcal{F}$ contains an $n$-ary function $f\notin \mathcal{E}(A)$ such that $f|_{A^n_{<n}}\in
\mathcal{E}(A)|_{A^n_{<n}}$, then $\mathcal{F}$ satisfies
$\Delta^{s}_{n}$.
\item If $\mathcal{F}$ contains a $\partial$-function, then $\mathcal{F}$ satisfies
$\Delta^{\partial}$.
\item If $\mathcal{F}$ contains a $\ell$-function, then $\mathcal{F}$ satisfies
$\Delta^{s}_{3}$.
\end{enumerate}
\end{claim}

\begin{proof} Let $f\in F_{[n]}\setminus \mathcal{E}(A)$ coincide with $e^n_i$ on $A^{n}_{<n}$. Then
$f(\mathbf{a})=a_j$ for some $\mathbf{a}=a_0a_1\ldots a_{n-1}\in
A^n_n$ and $j\in n\setminus\{\,i\,\}$. For any $k\in
n\setminus\{\,i,j\,\}$ denote the transposition $(a_j, a_k)\in S_A$
by $\sigma_{k}$, and the transposition $(j,k)$ by $\tau_k$. For any
$k<n$ denote
\[
s^k=
\begin{cases}
f_{\sigma_k}(x_{\tau_k(0)}, x_{\tau_k(1)}, \ldots,
x_{\tau_k(n-1)}),&\text{if $k\in n\setminus\{\,i,j\,\}$},\\
f,& \text{if $k=j$},\\
e^n_{i},&\text{if $k=i$}.
\end{cases}
\]

For all $k< n$ we have

\[
s^k(\mathbf{a})=a_k\,\,\text{and}\,\,s^k|_{A^{n}_{< n}}=
e^{n}_{i}|_{A^{n}_{< n}}.
\]

Let $\mathbf{b}$ be an arbitrary $n$-tuple from $A^{n}_{n}$ and
$\sigma_{\mathbf{b}}$ an arbitrary permutation of $A$ for which
$\sigma_{\mathbf{b}} (\mathbf{b})=\mathbf{a}$. Then for all $k<n$ we
have

\[
s^{k}_{\sigma_{\mathbf{b}}}(\mathbf{b})=b_k\,\,\text{and}\,\,s^{k}_{\sigma_{\mathbf{b}}}|_{A^{n}_{<
n}}= e^{n}_{i}|_{A^{n}_{< n}}.
\]

All the functions $s^{k}_{\sigma_{\mathbf{b}}}$ belong to
$\mathcal{F}$. Therefore, $\Delta^{e}_{n}$ holds.

\par Now let $\partial$ be a $\partial$-function from $\mathcal F$, and $\mathbf{a}=a_0a_1a_2$ an arbitrary triple from $A^{3}_{3}$.
Let $\partial(\mathbf{a})=a_i$ and $j\in 3\setminus \{\,i\,\}$.
Denote the transposition $(a_i, a_j)\in S_A$ by $\sigma_{j}$, and
the transposition $(i,j)$ by $\tau_j$. Let
$\partial^j=\partial_{\sigma_j}(x_{\tau_j(0)}, x_{\tau_j(1)},
x_{\tau_j(2)})$. Obviously, $\partial^j\in \mathcal{F}$. Moreover,
it easy to check that $\partial^j$ is a $\partial$-function and
$\partial^j(\mathbf{a})=a_j$. Therefore, $\Delta^{\partial}$ holds.

\par Arguing similarly, we find that if $\mathcal{F}$ contains an
$\ell$-function, then for any $\mathbf{a}=a_0a_1a_2\in A^{3}_{3}$
and $a\in \mathrm{ran}\, \mathbf{a}$ there is an $\ell$-function
$\ell\in \mathcal{F}_{[3]}$ such that $\ell(\mathbf{a})=a$. We show
that for any $i\in\{\,0,1,2\,\}$ there exists a function $s_i\in
\mathcal{F}_{[3]}$ such that
\begin{align*}
s_i(\mathbf{a})=a_i\,\,\text{and}\,\,s(\mathbf{x})=x_0\,\,\text{for
all $\mathbf{x}=x_0x_1x_{2}\in A^{3}_{<3}$}.
\end{align*}
\par If $i=0$, we put $s_0=e^{3}_{0}$. If $i\neq 0$, choose $\ell$-functions $\ell_0, \ell_i\in \mathcal{F}$ such
that $\ell_0(\mathbf{a})=a_0$ and $\ell_i(\mathbf{a})=a_i$. It is
easy to verify that the function
\[
s_i=\ell_0(x_0, \ell_{0}, \ell_{i})
\]
satisfies the required conditions.
\end{proof}

Now, for $r\geq 3$, Lemma \ref{RandSim} follows immediately from
Lemma \ref{sigma} and Claim \ref{simDelty}.

\par Before considering the Case $r=2$, we prove several auxiliary assertions.

\par For any pair $(\mathbf{a},\mathbf{b})\in A^{2}_{2}\times
A^{2}_{2}$ we define the \emph{type}
$\mathrm{t}(\mathbf{a},\mathbf{b})\in 2\cup 2^2\cup \{\,2\,\}$ of
$(\mathbf{a},\mathbf{b})$ as follows. Let $\mathbf{a}=a_0a_1$ and
$\mathbf{b}=b_0b_1$. Then for all $i,j\in \{\,0,1\,\}$
$$
\mathrm{t}(\mathbf{a},\mathbf{b})=
\begin{cases}
0,&\text{if $\mathbf{a}=\mathbf{b}$}\\
1,&\text{if $a_0=b_1$ and $a_1=b_0$}\\
ij,&\text{if $a_i=b_j$ and $a_{1-i}\neq b_{1-j}$}\\
2,&\text{if $\mathrm{ran}  \mathbf{a}\cap \mathrm{ran}  \mathbf{b}=\emptyset$}\\
\end{cases}
$$
Obviously,  the type $\mathrm{t}(\mathbf{a},\mathbf{b})$ is defined
for every $(\mathbf{a},\mathbf{b})\in A^{2}_{2}\times A^{2}_{2}$.
\par For any $i\in\{\,0,1\,\}$ we denote the binary relation $\rhd_i$ on
$A^{2}_{2}$ by
\[
\mathbf{a}\rhd_i \mathbf{b}\leftrightarrow \left(\left(\forall f\in
\mathcal{F}_{[2]}\right) f(\mathbf{a})=a_i\rightarrow
f(\mathbf{b})=b_i\right)
\]
for all $\mathbf{a}=a_0a_1,\mathbf{b}=b_0b_1\in A^{2}_{2}$.
\par For any pair $\mathbf{a}=a_0a_1\in A^{2}_{2}$, the symbol $\overline{
\mathbf{a}}$ denote the pair $a_1a_0$.

\begin{claim}\label{otnTriangl} For every  $i\in\{\,0,1\,\}$ the binary relation $\rhd_i$ is  reflexive and transitive.
Besides, for every $i\in\{\,0,1\,\}$, pairs
$\mathbf{a},\mathbf{b},\mathbf{a'}, \mathbf{b'}\in A^{2}_{2}$  and
permutation $\sigma\in S_{A}$
\begin{enumerate}
\item\label{1otnTriangl} $\mathbf{a}\rhd_i
\mathbf{b}\Rightarrow \sigma (\mathbf{a})\rhd_i\sigma(\mathbf{b})$,
\item\label{2otnTriangl} $\mathbf{a}\rhd_i \mathbf{b}\Rightarrow
\overline{\mathbf{a}}\rhd_{1-i} \overline{\mathbf{b}}$,
\item\label{3otnTriangl} $\mathbf{a}\rhd_i \mathbf{b}\Rightarrow
\mathbf{b}\rhd_{1-i}\mathbf{a}$,
\item\label{4otnTriangl} $t(\mathbf{a}, \mathbf{b})= t(\mathbf{a'},
\mathbf{b'})\Rightarrow (\mathbf{a}\rhd_i \mathbf{b}\rightarrow
\mathbf{a'}\rhd_i \mathbf{b'})$.
\end{enumerate}
\end{claim}

\begin{proof}
The reflexivity and transitivity of $\rhd_0$ and $\rhd_1$ are
obvious.

\par Let $\mathbf{a}=a_0a_1$, $\mathbf{b}=b_0b_1$ and $\mathbf{a}\rhd_i
\mathbf{b}$. Without loss of generality, $i=0$.

\par Assume
$f(\sigma(\mathbf{a}))=\sigma(a_0)$ for some $f\in
\mathcal{F}_{[2]}$ and $\sigma\in S_A$. Then we have $\sigma
(f_\sigma(\mathbf{a}))= f(\sigma(\mathbf{a}))=\sigma(a_0)$ and, so,
$f_\sigma(\mathbf{a})=a_0$, whence $f(\sigma(\mathbf{b}))=\sigma
(f_\sigma(\mathbf{b}))=\sigma(b_0)$. Item \ref{1otnTriangl} is
proved.

\par Assume $f(\overline{\mathbf{a}})=f(a_1a_0)=a_0$ for some $f\in
\mathcal{F}_{[2]}$. Let $g=f(x_1, x_0)$. Then we have
$g(a_0a_1)=f(a_1a_0)=a_0$ and, so,
$f(\overline{\mathbf{b}})=f(b_1b_0)=g(b_0b_1)=b_0$.
Item~\ref{2otnTriangl} is proved.

\par Assume $f(\mathbf{b})=f(b_0b_1)=b_1$ for some $f\in
\mathcal{F}_{[2]}$. If $f(\mathbf{a})=f(a_0a_1)=a_0$, we have
$f(\mathbf{b})=b_0$, a contradiction. So, $f(\mathbf{a})=a_1$. Item
\ref{3otnTriangl} is proved.

\par Item \ref{4otnTriangl} follows immediately from item
\ref{1otnTriangl}.

\end{proof}

\begin{claim}\label{Del2Triangl}  One of the following five cases holds.
\begin{enumerate}
\item\label{1Del2Triangl} $(\forall i<2)\,(\forall \mathbf{x},\mathbf{y}\in
A^{2}_{2})\,\mathbf{x}\rhd_i \mathbf{y}$,
\item\label{2Del2Triangl} $(\forall i<2)\,(\forall \mathbf{x},\mathbf{y}\in
A^{2}_{2})\,\mathbf{x}\rhd_i \mathbf{y}\leftrightarrow
\mathrm{t}(\mathbf{x},\mathbf{y})=0$,
\item\label{3Del2Triangl} $(\forall i<2)\,(\forall \mathbf{x},\mathbf{y}\in
A^{2}_{2})\,\mathbf{x}\rhd_i \mathbf{y}\leftrightarrow
\mathrm{t}(\mathbf{x},\mathbf{y})\in\{\,0,1\,\}$,
\item\label{4Del2Triangl} $|A|=4\wedge(\forall i<2)\,(\forall \mathbf{x},\mathbf{y}\in
A^{2}_{2})\,\mathbf{x}\rhd_i \mathbf{y} \leftrightarrow
\mathrm{t}(\mathbf{x},\mathbf{y})\in\{\,0,1,2\,\}$,
\item\label{5Del2Triangl} $|A|=3\wedge(\forall i<2)\,(\forall \mathbf{x},\mathbf{y}\in
A^{2}_{2})\,\mathbf{x}\rhd_i \mathbf{y} \leftrightarrow
\mathrm{t}(\mathbf{x},\mathbf{y})\in\{\,0,01,10\,\}$.
\end{enumerate}
\end{claim}
\begin{proof} Let $i$ be a fixed number in $\{\,0,1\,\}$.

\par Let $\rhd_i$ contain some pair $(\mathbf{a},\mathbf{b})\in A^{2}_{2}$
of type $00$. Let $\mathbf{a}=a_0a_1$ and $\mathbf{b}=a_0b_1$. Then
we have
\par $(a)$ $a_0a_1\rhd_{i}a_0b_1$ (supposition),
\par $(b)$ $a_0b_1\rhd_{1-i}a_0a_1$ from $(a)$ by item \ref{2otnTriangl}
of Claim \ref{otnTriangl},
\par $(c)$ $b_1a_0\rhd_{i}a_1a_0$ from $(b)$ by item
\ref{3otnTriangl} of Claim \ref{otnTriangl},
\par $(d)$ $a_0b_1\rhd_{i}a_1b_1$ from $(c)$ by item \ref{4otnTriangl}
of Claim \ref{otnTriangl},
\par $(e)$ $a_0a_1\rhd_{i} a_1b_1$ from $(a)$ and $(c)$ by transitivity,
\par $(f)$ $a_1b_1\rhd_{i} b_1a_0$ from $(e)$ by item \ref{4otnTriangl}
of Claim \ref{otnTriangl},
\par $(g)$ $a_0a_1\rhd_{i} b_1a_0$ from $(a)$ and $(f)$ by transitivity,
\par $(h)$ $a_0b_1\rhd_{i} a_1a_0$ from $(g)$ by item \ref{4otnTriangl}
of Claim \ref{otnTriangl},
\par $(i)$ $a_0a_1\rhd_{i} a_1a_0$ from $(a)$ and $(h)$ by transitivity.
\par The pairs from $(c), (e), (g), (i)$ have the types $11$,
$10$, $01$ and $1$, respectively. Given the reflexivity of $\rhd_i$
and item \ref{4otnTriangl} of Claim \ref{otnTriangl}, we have
\[
\mathbf{x}\rhd_i\mathbf{y}\,\,\text{for all
$(\mathbf{x},\mathbf{y})$ such that
$\mathrm{t}(\mathbf{x},\mathbf{y})\neq 2$}.
\]
If $|A|=3$, we have the case \ref{1Del2Triangl}. If $|A|\geq 4$,
choose $c\in A\setminus \{\,a_0, a_1, b_1\,\}$, and continue.
\par $(j)$ $a_0b_1\rhd_{i} b_1c$ from $(e)$ by item \ref{4otnTriangl}
of Claim \ref{otnTriangl},
\par $(k)$ $a_0a_1\rhd_{i} b_1c$ from $(a)$ and $(j)$ by transitivity.
\par  The pair $(a_0a_1, b_1c)$ has the type $2$. So,  we have the case \ref{1Del2Triangl}.

\par If $\rhd_i$ contains some pair $(\mathbf{a},\mathbf{b})\in A^{2}_{2}$
of type $11$, arguments are similar. Furthermore, we assume that
$\rhd_i$ does not contain pairs $(\mathbf{a},\mathbf{b})$ of
type~$00$ or~$11$.

\par Let $\rhd_i$ contain some  pair $(\mathbf{a},\mathbf{b})\in
A^{2}_{2}$ of type $01$. Let $\mathbf{a}=a_0a_1$ and
$\mathbf{b}=a_1b_1$. We have
\par $(a')$ $a_0a_1\rhd_{i}a_1b_1$ (supposition),
\par $(b')$ $a_1b_1\rhd_{i}b_1a_0$ from $(a')$ by item \ref{4otnTriangl}
of Claim \ref{otnTriangl},
\par $(c')$ $a_0a_1\rhd_{i} b_1a_0$ from $(a')$ и $(b')$ by transitivity.
\par So, $\mathbf{x}\rhd_i\mathbf{y}$ for all $(\mathbf{x},\mathbf{y})$ of type $10$.

\par If $\mathbf{x}\rhd_i\mathbf{y}$ for some pair $(\mathbf{x},\mathbf{y})$ of type $1$,
we have
\par $(d')$ $a_1a_0\rhd_{i}a_0a_1$ (supposition),
\par $(e')$ $a_1a_0\rhd_{i}a_1b_1$ from $(a')$ and $(d')$ by
transitivity.
\par A contradiction because $\mathrm{t}(a_1a_0,a_1b_1)=00$.

\par Therefore, $\rhd_i$ contains all pairs
of types $0, 01, 10$, and does not contain  any pair of types $1,
00, 11$. If $|A|=3$ we have a case \ref{5Del2Triangl}. If $|A|\geq
4$, choose $c\in A\setminus \{\,a_0, a_1, b_1\,\}$, and continue.
\par $(f')$ $a_1b_1\rhd_{i}b_1c$ from $(a')$ by item \ref{4otnTriangl}
of Claim \ref{otnTriangl},
\par $(g')$ $a_0a_1\rhd_{i}b_1c$ from $(a')$ and $(f')$ by transitivity,
\par $(h')$ $b_1c\rhd_{i}a_1a_0$ from $(g')$ by item \ref{4otnTriangl}
of Claim \ref{otnTriangl},
\par $(i')$ $a_0a_1\rhd_{i}a_1a_0$ from $(a')$ and $(h')$ by transitivity.
\par A contradiction because
$\mathrm{t}(a_0a_1, a_1a_0)=1$.

\par If $\rhd_i$ contains some pair of type $10$, arguments are
similar. Now we can assume that all pairs
$(\mathbf{x},\mathbf{y})\in \rhd_i$ have the type $0$, $1$, or $2$.

\par Let $\rhd_i$ contain some pair $(\mathbf{a},\mathbf{b})\in
A^{2}_{2}$ of type $2$. Let $\mathbf{a}=a_0a_1$ and
$\mathbf{b}=b_0b_1$. We have
\par $(a'')$ $a_0a_1\rhd_{i}b_0b_1$ (supposition),
\par $(b'')$ $b_0b_1\rhd_{i}a_1a_0$ from $(a'')$ by item \ref{4otnTriangl}
of Claim \ref{otnTriangl},
\par $(c'')$ $a_0a_1\rhd_{i}a_1a_0$ from $(a'')$ and $(b'')$ by transitivity.

\par Note that $\mathrm{t}(a_0a_1,a_1a_0)=1$. If $|A|=4$, we have the case \ref{4Del2Triangl}. If $|A|\geq
5$, choose $c\in A\setminus\{\,a_0,a_1,b_0,b_1\,\}$, and continue.
\par $(d'')$ $b_0b_1\rhd_{i}ca_0$ from $(a'')$ by item \ref{4otnTriangl}
of Claim \ref{otnTriangl},
\par $(e'')$ $a_0a_1\rhd_{i}ca_0$ from $(a'')$ and $(d'')$ by transitivity, a contradiction because
$\mathrm{t}(a_0a_1,ca_0)=01$.

\par Now we can assume that all pairs
$(\mathbf{x},\mathbf{y})\in \rhd_i$ have the type $0$ or $1$. By
item \ref{4otnTriangl} of Claim \ref{otnTriangl} we have that one of
the cases \ref{2Del2Triangl},~\ref{3Del2Triangl} holds.
\end{proof}
\begin{claim}\label{rhdequal} $\rhd_0=\rhd_1=\mathrm{R}_2(\mathcal{F})$.
\end{claim}
\begin{proof} By Claim \ref{Del2Triangl}, we have $(\mathbf{x},\mathbf{y})\in \rhd_i\Leftrightarrow (\mathbf{y},\mathbf{x})\in
\rhd_i$. It remains to use item \ref{3otnTriangl} of Claim
\ref{otnTriangl}.
\end{proof}

\par We continue the proof of Lemma \ref{RandSim}. Let $r=2$. Consider the cases of Claim~\ref{Del2Triangl}, taking into account Claim \ref{rhdequal}.
The case \ref{1Del2Triangl} implies   $r\geq 3$, a contradiction.
The cases \ref{2Del2Triangl} and \ref{3Del2Triangl} imply
$\Delta^2$.

\par Let the case \ref{4Del2Triangl} holds. We have $\mathrm{R}_2(\mathcal{F})=R_{\pm}(A)$ by Claim \ref{rhdequal}.
Choose an arbitrary set $B\in[A]^3$. It's easy to see that the clone
$\mathcal{G}=\mathcal F\upharpoonright B^{<\omega}$ satisfies
$\Delta^2$. Since $\mathcal{G}_{[3]}\in
\mathrm{Inv}_{A^3}\mathcal{G}$, Theorem \ref{d2} implies that
$\mathcal G$ contains some ternary function $f$ for which
\[
f|_{B^{3}_{<3}}=e^{3}_{0}|_{B^{3}_{<3}}\,\,\text{and}\,\,f(\mathbf{a})=a_1
\]
for some $\mathbf{a}=a_0a_1a_2\in B^{3}_{3}$.

\par  For any $\mathbf{b}\in A^{3}_{2}\setminus
B^{3}_{2}$ there is $\mathbf{c}\in B^{3}_{2}$ and $\sigma\in S_A$
such that $\sigma (\mathbf{b}) =\mathbf{c}$. Without loss of
generality we take $\mathbf{b}=b_0b_0b_1$ and
$\mathbf{c}=c_0c_0c_1$. Denote $f'=f(x_0, x_0, x_{1})$. Since
$b_0b_1\rhd_ic_0c_1$ for any $i\in\{\,0,1\,\}$, we have
\[
f(b_0b_0b_1)=b_0\leftrightarrow f'(b_0b_1)=b_0\leftrightarrow
f'(c_0c_1)=c_0\leftrightarrow f(c_0c_0c_1),
\]
so, $f$ coincides with the projection $e^{3}_{0}$ on the whole set
$A^{3}_{<3}$. It remains to use Claim \ref{simDelty}.

\par Finally, let the case \ref{5Del2Triangl} holds. We have $\mathrm{R}_2(\mathrm{F})=R_{\uparrow}(A)$ by Claim \ref{rhdequal}.
It is easy to see that each binary function $f\in \mathcal
F\setminus \mathcal{E}(A)$ is uniquely determined by its value on
any $\mathbf{a}\in A^{2}_{2}$. Let  $A=\{\,a,b,c\,\}$. Then
$\mathcal F$ contains exactly two functions $u$ и $v$ that are not
projections defined by

\begin{table}[h]
\centering
\begin{tabular}{c|ccc}
  $u$      &              $a$ & $b$ & $c$ \\\hline
  $a$          &          $a$ & $b$ & $a$ \\
  $b$          &          $b$ & $b$ & $c$ \\
  $c$          &          $a$ & $c$ & $c$ \\
   \end{tabular}
\qquad\qquad
\begin{tabular}{c|ccc}
  $v$      &              $a$ & $b$ & $c$ \\\hline
  $a$          &          $a$ & $a$ & $c$ \\
  $b$          &          $a$ & $b$ & $b$ \\
  $c$          &          $c$ & $b$ & $c$ \\
   \end{tabular}
\end{table}
\par Consider the function
\[
f=v(v(u(x_0,x_1),u(x_0,x_2)),u(x_1,x_2)).
\]
It easy to check that $f$ is $\partial$-function (for this it is
convenient to note that both functions $u$ and $v$ are commutative
and $u(x,y)=x\leftrightarrow v(x,y)=y$). It is remain to use Claim
\ref{simDelty}.
\end{proof}
\par Now the proof of the theorem reduces to an analysis of the cases of
Lemma \ref{RandSim}. We will collect some useful observations in one
claim.
\begin{claim}\label{Invan} Let $A$ and $Q$ be non-empty sets.
\begin{enumerate}
\item\label{Invan1} For any clone $\mathcal{F}\subseteq \mathcal{O}(A)$ and natural number $n$, $\mathcal{F}_{[n]}\in
\mathrm{Inv}_{A^n}\mathcal{F}$.
\item\label{Invan2} For any sets $H\subseteq A^Q$ and $P\in [Q]^1$, the
set $H|_P$ belongs to $\mathrm{Inv}_P\mathcal{F}$ for any
conservative clone $\mathcal{F}\subseteq \mathcal{O}(A)$.  Moreover,
for all conservative clones $\mathcal{F}, \mathcal{G}\subseteq
\mathcal{O}(A)$,
$[A^n]^1_{\mathcal{F}_{[n]}}=[A^n]^1_{\mathcal{G}_{[n]}}=\{\,\{\,\textbf{x}\,\}\colon
\textbf{x}\in A^n\,\}$, and $\mathcal{F}_n|_{P}=\mathcal{G}_n|_{P}$
for all $P\in [A^n]^1_{\mathcal{F}_{[n]}}$.
\item\label{Invan2a} For all clones $\mathcal{F}, \mathcal{G}\subseteq
\mathcal{C}(A)$ and a set $H\subseteq A^Q$ if
\mbox{$\mathcal{F}|_{A^{<\omega}_{<r}}=\mathcal{G}|_{A^{<\omega}_{<r}}$}
and \mbox{$(\forall q\in Q)\,|H(q)|<r$}, then $H\in
\mathrm{Inv}_Q\mathcal{F}\Leftrightarrow H\in
\mathrm{Inv}_Q\mathcal{G}$.
\item\label{Invan2b} For any conservative clone $\mathcal{F}\subseteq
\mathcal{O}(A)$, set $B\subseteq A$ and set $H\subseteq B^Q$, $H\in
\mathrm{Inv}_Q\mathcal{F}\Leftrightarrow H\in
\mathrm{Inv}_Q\mathcal{F}|_{B^{<\omega}}$.
\item\label{Invan3} For any clone $\mathcal{F}\subseteq \mathcal{O}(A)$ and natural number
$n$,
\[
\left[A^n\right]^{2,0}_{\mathcal{F}_{[n]}}=\{\,\{\,\mathbf{x},\mathbf{y}\,\}\in
\left[A^n\right]^2\colon  \mathbf{x}\neq
\mathbf{y}\,\text{and}\,(\mathbf{x},\mathbf{y})\in
\mathrm{R}_n(\mathcal{F})\,\}.
\]
Moreover, for all $\mathbf{x}=x_0x_1\ldots x_{n-1}$ and
$\mathbf{y}=y_0,y_1\ldots y_{n-1}$ in $A^n$, if
\mbox{$(\mathbf{x},\mathbf{y})\in \mathrm{R}_n(\mathcal{F})$}, then
\[
f(\mathbf{x})=x_i\Leftrightarrow f(\mathbf{y})=y_i
\]
for any $f\in \mathcal{F}_{[n]}$ and $i<n$. Therefore, for all
conservative clones \mbox{$\mathcal{F}, \mathcal{G}\subseteq
\mathcal{O}(A)$} if
$\mathrm{R}_{n}(\mathcal{F})=\mathrm{R}_{n}(\mathcal{G})$, then
$\left[A^n\right]^{2,0}_{\mathcal{F}_{[n]}}=\left[A^n\right]^{2,0}_{\mathcal{G}_{[n]}}$
and for all $P\in \left[A^n\right]^{2,0}_{\mathcal{F}_{[n]}}$,
$\mathcal{F}_{[n]}|_{P}=\mathcal{G}_{[n]}|_{P}$.
\item\label{Invan4} For any clone $\mathcal{F}\subseteq \mathcal{O}(A)$ and natural number
$n$,
\[
\left[A^n\right]^{2,1}_{\mathcal{F}_{[n]}}=\{\,\{\,\mathbf{x},\mathbf{y}\,\}\in
\left[A^n\right]^2\colon  \mathbf{x}\neq
\mathbf{y}\,\text{and}\,(\mathbf{x},\mathbf{y})\in
\mathrm{D}_n(\mathcal{F})\,\}.
\]
Moreover, if $\mathcal{F}$ is a conservative clone, then  for all
$\mathbf{x}$ and $\mathbf{y}$ in $A^n$ if
\mbox{$(\mathbf{x},\mathbf{y})\in\mathrm{D}_n(\mathcal{F})\setminus
\mathrm{R}_n(\mathcal{F})$}  and $\max(|\mathrm{ran}\,\mathbf{x}|,
|\mathrm{ran}\,\mathbf{y}|)\geq 2$, then
\begin{enumerate}
\item there is the unique pair $(a,b)\in A^2$ for which
$f(\mathbf{x})=a\vee f(\mathbf{y})=b$, and
\item $\mathcal{F}_{n}|_{\{\,\mathbf{x},\mathbf{y}\,\}}$ contains all
functions $f\in A^{\{\,\mathbf{x},\mathbf{y}\,\}}$ satisfying
\mbox{$f(\mathbf{x})\in \mathrm{ran}\,\mathbf{x}$},
$f(\mathbf{y})\in \mathrm{ran}\, \mathbf{y}$, and
$f(\mathbf{x})=a\vee f(\mathbf{y})=b$.
\end{enumerate}
Therefore, for all conservative clones $\mathcal{F},
\mathcal{G}\subseteq \mathcal{O}(A)$ if
$\mathrm{D}_{n}(\mathcal{F})=\mathrm{D}_{n}(\mathcal{G})$, then
$\left[A^n\right]^{2,1}_{\mathcal{F}_{[n]}}=\left[A^n\right]^{2,1}_{\mathcal{G}_{[n]}}$
and $\mathcal{F}_{[n]}|_{P}=\mathcal{G}_{[n]}|_{P}$ for all
\mbox{$P\in \left[A^n\right]^{2,0}_{\mathcal{F}_{[n]}}\setminus
\left[A^n\right]^{2,1}_{\mathcal{F}_{[n]}}$}.

\item\label{Invan5} For any conservative clone $\mathcal{F}\subseteq \mathcal{O}(A)$
with $\mathrm{r}(\mathcal{F})=3$ and $\Pi_0(\mathcal{F})=L_4$,
$\mathcal{F}|_{A^{<\omega}_{<3}}=\mathcal{L}(A)|_{A^{<\omega}_{<3}}$
where $\mathcal{L}(A)$ is the clone generated by all conservative
functions $\ell\in \mathcal{O}(A)_{[3]}$ satisfying
$\ell(x,y,y)=\ell(y,x,y)=\ell(y,y,x)=x$.
\item\label{Invan6} For any conservative symmetric clone $\mathcal{F}\subseteq
\mathcal{O}(A)$, if $|A|=4$, \mbox{$\mathrm{r}(\mathcal{F})=2$} and
$\mathrm{R}_2(\mathcal{F})=R_{\pm}$, then
$\mathcal{F}|_{A^{<\omega}_{<3}}=\mathcal{K}(A)|_{A^{<\omega}_{<3}}$
where $\mathcal{K}(A)$ is the set (in~fact, the clone) of all
conservative Klein's $\Pi_0(\mathcal{F})$-functions.
\end{enumerate}
\end{claim}
\begin{proof}By a direct verification.
\end{proof}

\par Case 1 ($\mathrm{r}(\mathcal{F})\geq 4$). By Lemma \ref{rparametr}
any function $f\in \mathcal{F}_{[n]}$ coincides with a projection on
$A^n_{<r}$. Therefore, by item \ref{Invan2a} of Claim \ref{Invan},
any set $H\subseteq A^P$ satisfying $(\forall p\in P)\, |H(p)|<r$
belongs to $\mathrm{Inv}_P\mathcal{F}$.  So, item \ref{rgeq4} of
Theorem \ref{main} follows from item \ref{Invan2} of Claim
\ref{Invan} and Theorem \ref{s3}. To prove the statement
$(\ast\ast\ast)$ in this case, it suffices to note that
\mbox{$(A^n)^{(r)}_{\mathcal{F}_{[n]}}=A^n_{<r}$} and
\mbox{$\mathcal{F}_{[n]}|_{A^n_{<r}}=\mathcal{E}_{[n]}|_{A^n_{<r}}$}
for any conservative clone $\mathcal{F}\subseteq \mathcal{O}(A)$
with $\mathrm{r}(\mathcal{F})\geq 4$, and use items~\ref{Invan1},
\ref{Invan2}, \ref{Invan3} of Claim \ref{Invan}, and Theorem
\ref{s3}.

\par Case 2 ($\mathrm{r}(\mathcal{F})= 3$ and $\Pi_0(\mathcal{F})\neq
L_4$). Item \ref{req3} of Theorem \ref{main} follows from item
\ref{Invan2} of Claim \ref{Invan} and Theorem \ref{partial}. The
statement $(\ast\ast\ast)$ follows from items \ref{Invan1},
\ref{Invan2}, \ref{Invan3}, \ref{Invan4} of Claim \ref{Invan}, and
Theorem \ref{partial}.

\par Case 3 ($\mathrm{r}(\mathcal{F})= 3$ and $\Pi_0(\mathcal{F})=
L_4$). Item \ref{req3L4} of Theorem \ref{main} follows from items
\ref{Invan2}, \ref{Invan2a}, \ref{Invan5} of Claim \ref{Invan} and
Theorem \ref{s3}. To prove the statement $(\ast\ast\ast)$ in this
case, it suffices to note that
$(A^n)^{(3)}_{\mathcal{F}_{[n]}}=A^n_{<3}$ for any conservative
clone \mbox{$\mathcal{F}\subseteq \mathcal{O}(A)$}, and use items
\ref{Invan1}, \ref{Invan2}, \ref{Invan3}, \ref{Invan5} of Claim
\ref{Invan}, and Theorem \ref{s3}.

\par Case 4 ($\mathrm{r}(\mathcal{F})= 2$ and the following cases do not
hold: (i) $|A|=4$ and $\mathrm{R}_2(\mathcal{F})=R_{\pm}$, (ii)
$|A|=3$ and $\mathrm{R}_2(\mathcal{F})=R_{\uparrow}$). Item
\ref{req2notspecial} of Theorem \ref{main} follows from items
\ref{Invan2}, \ref{Invan2b} of Claim \ref{Invan} and Theorem
\ref{d2}. To prove the statement $(\ast\ast\ast)$ in this case, note
that for any conservative symmetric clones \mbox{$\mathcal{F},
\mathcal{G}\subseteq \mathcal{O}(A)$} and set $B\in [A]^2$,
$(A^n)^{[B]}_{\mathcal{F}_{[n]}}=(A^n)^{[B]}_{\mathcal{G}_{[n]}}=B^n$
and, if $\Pi_0(\mathcal{F})=\Pi_0(\mathcal{G})$,
\mbox{$\mathcal{F}_{[n]}|_{B^n}=\mathcal{G}_{[n]}|_{B^n}$}. Besides,
$[A^n]^{2,id}_{\mathcal{F}_{[n]}}=\emptyset$ for any clone
$\mathcal{F}\subseteq \mathcal{O}(A)$. Now it is sufficient to use
items \ref{Invan1}, \ref{Invan2} of Claim \ref{Invan}, and Theorem
\ref{d2}.

\par Case 5 ($\mathrm{r}(\mathcal{F})= 2$, $|A|=4$ and $\mathrm{R}_2(\mathcal{F})=R_{\pm}$).
Item \ref{req2A4} of Theorem~\ref{main} follows from items
\ref{Invan2}, \ref{Invan2a}, \ref{Invan6} of Claim \ref{Invan} and
Theorem \ref{s3}. To prove the statement $(\ast\ast\ast)$ in this
case, it suffices to note that
$(A^n)^{(3)}_{\mathcal{F}_{[n]}}=A^n_{<3}$ for any conservative
clone $\mathcal{F}\subseteq \mathcal{O}(A)$, and use items
\ref{Invan1}, \ref{Invan2}, \ref{Invan3}, \ref{Invan6} of Claim
\ref{Invan}, and Theorem \ref{s3}.

\par Case 6 ($\mathrm{r}(\mathcal{F})= 2$, $|A|=3$ and $\mathrm{R}_2(\mathcal{F})=R_{\uparrow}$).
Item \ref{req2A3} of Theorem \ref{main} follows from item
\ref{Invan2} of Claim \ref{Invan} and Theorem \ref{partial}. The
statement $(\ast\ast\ast)$ follows from items \ref{Invan1},
\ref{Invan2}, \ref{Invan3}, \ref{Invan4}, and Theorem \ref{partial}.

\end{proof}
\textbf{Remarks}. The classification of symmetric conservative
clones $\mathcal{F}$ with a finite carrier can be further detailed,
since the relations $\mathrm{R}(\mathcal{F})$ and
$\mathrm{D}(\mathcal{F})$ have a very special form in this case, see
\cite{Polyakov2016}. The parameter $\mathrm{r}(\mathcal{F})$ is
introduced in \cite{Shelah2005}. Symmetric clones with a finite
carrier containing all constants are described in \cite{Xoa} (note
that conservative clones, on the contrary, do not contain any
constants).



\begin{thebibliography}{99}

\bibitem{Csakany86}
    Cs\'{a}k\'{a}ny B.:
    On conservative minimal operations.
    Lect. in Univ. Alg., Proc. Conf., Colloq. Math. Soc. Janos Bolyai,
    Szeged.
    \textbf{43}, 49--60 (1986)

\bibitem{Jezek95}
    Je\v{z}ek J., Quackenbush R.:
    Minimal clones of conservative functions.
    Int. J. of Alg. and Comp.
    \textbf{05}(06), 615--630 (1995)

\bibitem{Shelah2005}
    Shelah S.:
    On the Arrow property.
    Adv. in Ap. Mat.
    \textbf{34}, 217--251 (2005)

\bibitem{Polyakov2014}
    Polyakov N.,  Shamolin M.:
    On a generalization of Arrow's impossibility theorem.
    Dokl. Math.
    \textbf{89}(3), 290--292 (2014)

\bibitem{Poshel}
    P\"{o}schel R., Kalu\v{z}nin L.:
    Funktionenund Relationenalgebren. Ein Kapitel der Diskreten Mathematik.
    Veb Deutscher Verlag Der Wissenschaften, Berlin (1979)

\bibitem{Lau}
    Lau D.:
    Function Algebras on Finite Sets. A Basic Course on Many-Valued Logic and Clone Theory.
    Springer-Verlag, Berlin Heidelberg (2006).

\bibitem{Post}
    Post E.:
    Two-valued iterative systems of mathematical logic.
    Vol. 5 of Annal of Math. studies, Princeton Univer (1942).

\bibitem{Marchen1997}
    Marchenkov S.:
    Clone classification of dually discriminator algebras with a finite carrier.
    Mat.  Zam.
    \textbf{61}(3), 359--366 (1997) (Russian)

\bibitem{Xoa}
    Nguen, V. K.:
    Families of closed classes of $k$-valued logic that are preserved by all automorphisms.
    Diskretn. Mat.
    \textbf{5}(4), 87--108 (1993) (Russian)


\bibitem{Polyakov2016}
    Polyakov, N.:
    Galois connections for classes of discrete functions and their application to mathematical problems of social
    choice theory.
    PhD thesis, Moscow University (2016)



\end{thebibliography}
\end{document}